\documentclass{article}%
\usepackage{amsmath}
\usepackage{amsfonts}
\usepackage{amssymb}
\usepackage{graphicx}
\usepackage{subfigure}
\usepackage{subfloat}

\providecommand{\U}[1]{\protect\rule{.1in}{.1in}}
\newtheorem{theorem}{Theorem}

\newtheorem{definition}[theorem]{Definition}

\newtheorem{lemma}[theorem]{Lemma}

\newenvironment{proof}[1][Proof]{\noindent\textbf{#1.} }{\ \rule{0.5em}{0.5em}}
\input{tcilatex}
\begin{document}

\author{J.C. Jimenez\thanks{Instituto de Cibernetica, Matematica y Fisica, Calle 15,
No. 551, entre C y D, Vedado, La Habana 10400, Cuba. e-mail: (see
http://sites.google.com/site/locallinearization/)}}
\title{Approximate linear minimum variance filters for continuous-discrete state
space models: convergence and practical algorithms}
\maketitle

\begin{abstract}
In this paper, approximate Linear Minimum Variance (LMV) filters for
continuous-discrete state space models are introduced. The filters
are obtained by means of a recursive approximation to the
predictions for the first two moments of the state equation. It is
shown that the approximate filters converge to the exact LMV filter
when the error between the predictions and their approximations
decreases. As particular instance, the order-$\beta$ Local
Linearization filters are presented and expounded in detail.
Practical algorithms are also provided and their performance in
simulation is illustrated with various examples. The proposed
filters are intended for the recurrent practical situation where a
nonlinear stochastic system should be identified from a reduced
number of partial and noisy observations distant in time.
\end{abstract}

\section{Introduction}

The estimation of unobserved states of a continuous stochastic
dynamical system from noisy discrete observations is of central
importance to solve diverse scientific and technological problems.
The major contribution to the solution of this estimation problem is
due to Kalman and Bucy \cite{Kalman 1961,Kalman 1963}, who provided
a sequential and computationally efficient solution to the optimal
filtering and prediction problem for linear state space models with
additive noise. However, the optimal estimation of nonlinear state
space models is still a subject of active researches. Typically, the
solution of optimal filtering problems involves the resolution of
evolution equations for conditional probabilistic densities, moments
or modes, which in general have explicit solutions in few particular
cases. Therefore, a variety approximations have been developed.
Examples of such approximate nonlinear filters are the classical
ones as the Extended Kalman, the Iterated Extended Kalman, the
Gaussian and the Modified Gaussian filters \cite{Jazwinski 1970};
and other relatively recents as the Local Linearization \cite{Ozaki
1993,Jimenez 2003}, the Projection \cite{Brigo 1999} and the
Particle filters \cite{Del Moral 2001} methods.

In a variety of practical situations, the solution of the general
optimal filtering problem is dispensable since the solution provided
by a suboptimal filter is satisfactory. This is the case of the
signal filtering and detection problems, the system stabilization,
and the parameter estimation of nonlinear systems, among others.
Prominent examples of suboptimal filters are the linear, the
quadratic and the polynomial one, which have been widely used for
the estimation of the state of both, continuous-continuous
\cite{Mil'shtein 1984,Mohler 1980,Phillis 1989} and
discrete-discrete \cite{De Koning 1984,Pakshin 1978,Phillis 1989,De
Santis 1995,Carravetta 1997} models. In the case of
continuous-discrete models, exact expressions for Linear Minimum
Variance filter (LMV) have also been derived \cite{Jimenez 2002},
but they are restricted to linear models. For nonlinear models, this
kind of suboptimal filter has in general no exact solution since the
first two conditional moments of the state equation has no explicit
solution. Therefore, adequate approximations are required in this
situation as well.

In this paper, approximate LMV filters for nonlinear
continuous-discrete state space models are introduced. The filters
are obtained by means of a recursive approximation to the
predictions for the first two moments of the state equation. It is
shown that the approximate filters converge to the exact LMV filter
when the error between the predictions and their approximations
decreases. Based on the well-known Local Linear approximations for
the state equation, the order-$\beta$ Local Linearization filters
are presented as a particular instance. Their convergence, practical
algorithms and performance in simulations are also considered in
detail. The simulations show that these Local Linearization filters
provide accurate and computationally efficient estimation of the
unobserved states of the stochastic systems given a reduced number
of partial and noisy observations, which is a typical situation in
practical control engineering.

The paper is organized as follows. In section 2, basic notations and
results on LMV filters, Local Linear approximations and Local
Linearization filters are presented. The general class of
approximate LMV filters is introduced in section 3 and its
convergence is stated. In section 4, the order-$\beta$ Local
Linearization filters are presented and their convergence analyzed.
In the last two sections, practical algorithms for these filters and
their performance in simulations are considered.

\section{Notations and Preliminaries}

Let $(\Omega,\mathcal{F},P)$ be a complete probability space, and
$\{\mathcal{F}_{t},$ $t\geq t_{0}\}$ be an increasing right
continuous family of complete sub $\sigma$-algebras of
$\mathcal{F}$. Consider the state space model defined by the
continuous state equation
\begin{equation}
d\mathbf{x}(t)=\mathbf{f}(t,\mathbf{x}(t))dt+\sum\limits_{i=1}^{m}\mathbf{g}%
_{i}(t,\mathbf{x}(t))d\mathbf{w}^{i}(t),\text{ }  \label{SS1}
\end{equation}%
for all $t\in \lbrack t_{0},T]$, and the discrete observation
equation
\begin{equation}
\mathbf{z}_{t_{k}}=\mathbf{Cx}(t_{k})+\mathbf{e}_{t_{k}},\text{ }
\label{SS2}
\end{equation}%
for all $k=0,1,..,M-1$, where $\mathbf{f}$, $\mathbf{g}_{i}:[t_{0},T]\times\mathbb{R}^{d}%
\rightarrow\mathbb{R}^{d}$ are functions, $\mathbf{w=(\mathbf{w}}^{1}%
,\ldots,\mathbf{w}^{m}\mathbf{)}$ is an $m$-dimensional $\mathcal{F}_{t}%
$-adapted standard Wiener process, $\{\mathbf{e}_{t_{k}}:\mathbf{e}_{t_{k}%
}\thicksim\mathcal{N}(0,\Sigma_{t_{k}}),$ $k=0,..,M-1\}$ is a sequence of $r$%
-dimensional i.i.d. Gaussian random vectors independent of
$\mathbf{w}$, $\Sigma_{t_{k}}$ an $r\times r$ positive semi-definite
matrix, and $\mathbf{C}$ an $r\times d$ matrix. Here, it is assumed
that the $M$ time
instants $t_{k}$ define an increasing sequence $\{t\}_{M}=\{t_{k}:t_{k}%
<t_{k+1}$, $t_{M-1}=T$, $k=0,1,..,M-1\}$. Conditions for the
existence and uniqueness of a strong solution of (\ref{SS1}) with
bounded moments are assumed.

Let $\mathbf{x}_{t/\rho}=E(\mathbf{x(}t)/Z_{\rho})$ and
$\mathbf{Q}_{t/\rho
}=E(\mathbf{x(}t)\mathbf{x}^{\intercal}(t)/Z_{\rho})$ be the first
two conditional moments of $\mathbf{x}$ with $\rho\leq t$, where
$E(.)$ denotes the mathematical expectation value, and $Z_{\rho
}=\{\mathbf{z}_{t_{k}}:$ $t_{k}\leq \rho ,$ $t_{k}\in \{t\}_{M}\}$
is a time series with observations from (\ref{SS2}). Further, let us
denote by
\begin{align*}
\mathbf{U}_{t/\rho}  &  =E((\mathbf{x(}t)-\mathbf{x}_{t/\rho})(\mathbf{x(}%
t)-\mathbf{x}_{t/\rho})^{\intercal}/Z_{\rho})\\
&  =\mathbf{Q}_{t/\rho}-\mathbf{x}_{t/\rho}\mathbf{x}_{t/\rho}^{\intercal}%
\end{align*}
the conditional variance of $\mathbf{x}$.

Denote by $\mathcal{C}_{P}^{l}(\mathbb{R}^{d},\mathbb{R})$ the space
of $l$ time continuously differentiable functions
$g:\mathbb{R}^{d}\rightarrow \mathbb{R}$ for which $g$ and all its
partial derivatives up to order $l$ have polynomial growth.

\subsection{Linear minimum variance filtering problem}

According to \cite{Battin 1962,Schmidt 1966,Sorenson 1966,Jazwinski
1970} the linear minimum variance filter
$\mathbf{x}_{t_{k+1}/t_{k+1}}$ for a state space model with discrete
observation equation (\ref{SS2}) is defined as
\[
\mathbf{x}_{t_{k+1}/t_{k+1}}=\mathbf{x}_{t_{k+1}/t_{k}}+\mathbf{G}_{t_{k+1}%
}\mathbf{(\mathbf{z}}_{t_{k+1}}-\mathbf{\mathbf{C}x}_{t_{k+1}/t_{k}}%
\mathbf{)},
\]
where the filter gain $\mathbf{G}_{t_{k+1}}$ is to be determined so
as to minimize the error variance
\[
E((\mathbf{x}(t_{k+1})-\mathbf{x}_{t_{k+1}/t_{k+1}})(\mathbf{x}(t_{k+1}%
)-\mathbf{x}_{t_{k+1}/t_{k+1}})^{\intercal}).
\]
This yields to the following definition.

\begin{definition}\label{LMV filter}
The Linear Minimum Variance filter for the state space model (\ref{SS1}%
)-(\ref{SS2}) is defined, between observations, by
\begin{equation}
\frac{d\mathbf{x}_{t/t}}{dt}=E(\mathbf{f}(t,\mathbf{x})/Z_{t})\label{LMVF1}%
\end{equation}%
\begin{align}
\frac{d\mathbf{U}_{t/t}}{dt} &  =E(\mathbf{xf}^{\intercal}(t,\mathbf{x}%
)/Z_{t})-\mathbf{x}_{t/t}E(\mathbf{f}^{\intercal}(t,\mathbf{x})/Z_{t}%
)+E(\mathbf{f}(t,\mathbf{x})\mathbf{x}^{\intercal}/Z_{t})\nonumber\\
&  -E(\mathbf{f}(t,\mathbf{x})/Z_{t})\mathbf{x}_{t/t}^{\intercal}-%
{\displaystyle\sum\limits_{i=1}^{m}}
E(\mathbf{g}_{i}(t,\mathbf{x})\mathbf{g}_{i}^{\intercal}(t,\mathbf{x}%
)/Z_{t})\label{LMVF2}%
\end{align}
for all $t\in(t_{k},t_{k+1})$, and by
\begin{equation}
\mathbf{x}_{t_{k+1}/t_{k+1}}=\mathbf{x}_{t_{k+1}/t_{k}}+\mathbf{G}_{t_{k+1}%
}\mathbf{(\mathbf{z}}_{t_{k+1}}-\mathbf{\mathbf{C}x}_{t_{k+1}/t_{k}}%
\mathbf{)}\label{LMVF3}%
\end{equation}%
\begin{equation}
\mathbf{U}_{t_{k+1}/t_{k+1}}=\mathbf{U}_{t_{k+1}/t_{k}}-\mathbf{G}_{t_{k+1}%
}\mathbf{CU}_{t_{k+1}/t_{k}}\label{LMVF4}%
\end{equation}
for each observation at $t_{k+1}$, with filter gain%
\begin{equation}
\mathbf{G}_{t_{k+1}}=\mathbf{U}_{t_{k+1}/t_{k}}\mathbf{C}^{\intercal
}{\Large (}\mathbf{CU}_{t_{k+1}/t_{k}}\mathbf{C}^{\intercal}+\Sigma_{t_{k+1}%
})^{-1}\label{LMVF5}%
\end{equation}
for all $t_{k},t_{k+1}\in\{t\}_{M}$. The predictions
$\mathbf{x}_{t/t_{k}}$, $\mathbf{U}_{t/t_{k}}$ are accomplished,
respectively, via expressions (\ref{LMVF1})-(\ref{LMVF2}) with
initial conditions $\mathbf{x}_{t_{k}/t_{k}}$ and
$\mathbf{U}_{t_{k}/t_{k}}$ for all $t\in(t_{k},t_{k+1}]$ and
$t_{k},t_{k+1}\in\{t\}_{M}$.
\end{definition}

Note that, in continuous-discrete filtering problem, the filters
$E(\mathbf{x}(t)/Z_{t})$ and $E(\mathbf{x}(t)\mathbf{x}^{\intercal
}(t)/Z_{t}) $ reduce to the predictions $E(\mathbf{x}(t)/Z_{t_{k}})$
and $E(\mathbf{x}(t)\mathbf{x}^{\intercal }(t)/Z_{t_{k}})$ for all
$t$ between two consecutive observations $t_{k}$ and $t_{k+1}$, that
is for all $t\in(t_{k},t_{k+1})$. This is because there is not more
observations between $t_{k}$ and $t_{k+1}$. This implies that, in
the above definition, $\mathbf{x}_{t_{k+1}-\varepsilon
/t_{k+1}-\varepsilon }$\bigskip $\equiv
\mathbf{x}_{t_{k+1}-\varepsilon /t_{k}}$ for all $\varepsilon >0$
and so $\mathbf{x}_{t_{k+1}-\varepsilon /t_{k+1}-\varepsilon }$
tends to $\mathbf{x}_{t_{k+1}/t_{k}}$ when $\varepsilon$ goes to
zero.

Clearly, for linear state equation with additive noise, the LMV
filter (\ref{LMVF1})-(\ref{LMVF5}) reduces to the classical
continuous-discrete Kalman filter. For linear state equation with
multiplicative noise, explicit formulas for the LMV filter can be
found in \cite{Jimenez 2002}. In general, since the
integro-differential equations (\ref{LMVF1})-(\ref{LMVF2}) of the
LMV filter have explicit solution for a few simple state equations,
approximations to them are needed. In principle, for this type of
suboptimal filter, the same conventional approximations to the
general optimal minimum variance filter may be used as well. For
instance, those for the solution of (\ref{LMVF1})-(\ref{LMVF2})
provided by the conventional Extended Kalman, the Iterated Extended
Kalman, the Gaussian, the Modified Gaussian and the Local
Linearization filters. However, in all these approximations, once
the data $Z_{t_{M}}$ are given on a time partition $\{t\}_{M}$ the
error between the exact and the approximate predictions for the mean
and variance of (\ref{SS1}) at $t_{k}$ is completely settled by
$t_{k}-t_{k-1}$ and can not be reduced. Therefore, small enough time
distance between consecutive observations would be typically
necessary to obtain an adequate approximation to the LMV filter.
Undoubtedly, this imposes undesirable restrictions to the time
distance between observations that can not be accomplished in many
practical situations. This drawback can be overcome by means of the
particle filter introduced in \cite{Del Moral 2001}, but at expense
of a very high computation cost. Note that this filter performs, by
means of intensive simulations, an estimation of the whole
probabilistic distribution of the processes $\mathbf{x}$ solution of
(\ref{SS1}) from which the first two conditional moments of
$\mathbf{x}$ can then be computed. Obviously, this general solution
to the filtering problem is not practical when an expedited
computation of the LMV filter (\ref{LMVF1})-(\ref{LMVF5}) is
required, which is typically demanded in many applications. For
example, the LMV filter and its approximations are a key component
of the innovation method for the parameter estimation of diffusion
processes from a time series of partial and noisy observations
\cite{Ozaki
1994,Shoji98,Nielsen01,Nielsen00a,Nielsen00b,Singer02,Jimenez06
JTSA}. For this purpose, accurate and computationally efficient
approximations to the LMV filter will be certainly usefull.

\subsection{Local Linearization filters}

A key component for constructing the Local Linearization (LL)
filters is the concept of Weak Local Linear (WLL) approximation for
Stochastic Differential Equations (SDEs) \cite{Jimenez02 SAA,Jimenez
2003}.

Let us consider the SDE (\ref{SS1}) on the time interval
$[a,b]\subset \lbrack t_{0},T]$, and the time discretization $\left(
\tau \right) _{h}=\left\{ \tau _{n}:n=0,1,\ldots ,N\right\} $ of
$[a,b]$ with maximum stepsize $h$ defined as a sequence of times
that satisfy the
conditions $a=\tau _{0}<\tau _{1}<\cdots <\tau _{N}=b$, and $\underset{n}{%
\max }(\tau _{n+1}-\tau _{n})\leq h<1$ for $n=0,\ldots ,N-1$. Further, let%
\begin{equation*}
n_{t}=\max \{n=0,1,\ldots ,N:\tau _{n}\leq t\text{ and }\tau _{n}\in
\left( \tau \right) _{h}\}
\end{equation*}%
for all $t\in \left[ a,b\right] $.

\begin{definition}
For a given time discretization $\left( \tau \right) _{h}$ of $%
\left[ a,b\right] $, the stochastic process $\mathbf{y}=\{\mathbf{y}(t),$ $%
t\in \left[ a,b\right] \}$ is called order-$\beta $ $(=1,2)$ Weak
Local Linear approximation of the solution of (\ref{SS1}) on $\left[
a,b\right] $\ if it is the weak solution of the piecewise linear
equation
\begin{equation}
d\mathbf{y}(t)=(\mathbf{A}(\tau _{n_{t}})\mathbf{y}(t)+\mathbf{a}^{\mathbb{%
\beta }}(t;\tau
_{n_{t}}))dt+\sum\limits_{i=1}^{m}(\mathbf{B}_{i}(\tau
_{n_{t}})\mathbf{y}(t)+\mathbf{b}_{i}^{\mathbb{\beta }}(t;\tau _{n_{t}}))d%
\mathbf{w}^{i}(t)  \label{WLLA}
\end{equation}%
for all $t\in (\tau _{n},\tau _{n+1}]$ and initial value $\mathbf{y}(a)=%
\mathbf{y}_{0}$, where the matrices functions $\mathbf{A,B}_{i}$ are
defined as
\begin{equation*}
\mathbf{A}(s)=\frac{\partial \mathbf{f}(s,\mathbf{y}(s))}{\partial \mathbf{y}%
}\text{ \ \ \ \ \ \ and \ \ \ \ \ \
}\mathbf{B}_{i}(s)=\frac{\partial
\mathbf{g}_{i}(s,\mathbf{y}(s))}{\partial \mathbf{y}},
\end{equation*}%
and the vectors functions $\mathbf{a}^{\mathbb{\beta }}$, $\mathbf{b}_{i}^{\mathbb{%
\beta }}$ as
\begin{equation*}
\mathbf{a}^{\beta }(t;s)=\left\{
\begin{array}{ll}
\mathbf{f}(s,\mathbf{y}(s))-\frac{\partial \mathbf{f}(s,\mathbf{y}(s))}{%
\partial \mathbf{y}}\mathbf{y}(s)+\frac{\partial \mathbf{f}(s,\mathbf{y}(s))%
}{\partial s}(t-s) & \text{for }\mathbb{\beta }=1 \\
\mathbf{a}^{1}(t;s)+\frac{1}{2}\sum\limits_{j,l=1}^{d}[\mathbf{G}(s,\mathbf{y%
}(s))\mathbf{G}^{\intercal }(s,\mathbf{y}(s))]^{j,l}\text{
}\frac{\partial
^{2}\mathbf{f}(s,\mathbf{y}(s))}{\partial \mathbf{y}^{j}\partial \mathbf{y}%
^{l}}(t-s) & \text{for }\mathbb{\beta }=2%
\end{array}%
\right.
\end{equation*}%
and%
\begin{equation*}
\mathbf{b}_{i}^{\beta }(t;s)=\left\{
\begin{array}{ll}
\mathbf{g}_{i}(s,\mathbf{y}(s))-\frac{\partial \mathbf{g}_{i}(s,\mathbf{y}%
(s))}{\partial \mathbf{y}}\mathbf{y}(s)+\frac{\partial \mathbf{g}_{i}(s,%
\mathbf{y}(s))}{\partial s}(t-s) & \text{for }\mathbb{\beta }=1 \\
\mathbf{b}_{i}^{1}(t;s)+\frac{1}{2}\sum\limits_{j,l=1}^{d}[\mathbf{G}(s,%
\mathbf{y}(s))\mathbf{G}^{\intercal }(s\mathbf{,y}(s))]^{j,l}\text{ }\frac{%
\partial ^{2}\mathbf{g}_{i}(s,\mathbf{y}(s))}{\partial \mathbf{y}%
^{j}\partial \mathbf{y}^{l}}(t-s) & \text{for }\mathbb{\beta }=2%
\end{array}%
\right.
\end{equation*}%
for all $s\leq t$. Here, $\mathbf{G=[g}_{1},\ldots ,\mathbf{g}_{m}]$ is an $%
d\times m$ matrix function.
\end{definition}

The drift and diffusion coefficients of the equation (\ref{WLLA})
are, respectively, weak approximations of order $\beta $ to the
drift and diffusion coefficients of the equation (\ref{SS1})
obtained from the
Ito-Taylor expansion of order $\beta $. That is \cite{Kloeden 1995},%
\begin{equation*}
\underset{s\leq t\leq s+h}{\sup }\left\vert E(g(\mathbf{f}(t,\mathbf{y}%
(t)))-E(g(\mathbf{A}(s)\mathbf{y}(t)+\mathbf{a}^{\mathbb{\beta }%
}(t;s)))\right\vert \leq Ch^{\beta }
\end{equation*}%
and
\begin{equation*}
\underset{s\leq t\leq s+h}{\sup }\left\vert E(g(\mathbf{g}_{i}(t,\mathbf{y}%
(t)))-E(g(\mathbf{B}_{i}(s)\mathbf{y}(t)+\mathbf{b}_{i}^{\mathbb{\beta }%
}(t;s)))\right\vert \leq Ch^{\beta }
\end{equation*}%
for all $h>0$ and $s\in \left[ a,b-h\right] $, where $g\in \mathcal{C}%
_{P}^{2(\beta +1)}(\mathbb{R}^{d},\mathbb{R})$ and $C$ is a positive
constant.

Explicit formulas for the conditional mean $\mathbf{y}_{t/\rho}$ and
variance $\mathbf{V}_{t/\rho}$ of $\mathbf{y}$ were initially given in \cite%
{Jimenez 2002,Jimenez 2003} and simplified later in \cite{Jimenez
2012}.

The conventional Local Linearization filters for the model (\ref{SS1})-(\ref%
{SS2}) are obtained in two steps \cite{Jimenez 2003}: 1) by
approximating the solution of the nonlinear state equation on each
time subinterval $[t_{k},t_{k+1}]$ by the Local Linear approximation
(\ref{WLLA}) on $[t_{k},t_{k+1}]$ with time discretization $\left(
\tau \right) _{h}\equiv \{t_{k},t_{k+1}\}$ for all $t_{k},t_{k+1}\in
\{t\}_{M}$; and 2) by the recursive application of the linear
minimum variance filter \cite{Jimenez 2002} to the resulting
piecewise linear continuous-discrete model. This yields to the
following.

\begin{definition}
Given a time discretization $\left(  \tau\right)
_{h}\equiv\{t\}_{M}$, the Local Linearization filter for the state
space model (\ref{SS1})-(\ref{SS2}) is defined, between
observations, by the linear equations
\begin{equation}
\frac{d\mathbf{y}_{t/t}}{dt}=\mathbf{A}(t_{n_{t}})\mathbf{y}_{t/t}%
+\mathbf{a}^{\beta}(t;t_{n_{t}}) \label{LLF1}%
\end{equation}%
\begin{equation}
\frac{d\mathbf{V}_{t/t}}{dt}=\mathbf{A}(t_{n_{t}})\mathbf{V}_{t/t}%
+\mathbf{V}_{t/t}\mathbf{A}^{\intercal}(t_{n_{t}})+\sum\limits_{i=1}%
^{m}\mathbf{B}_{i}(t_{n_{t}})\mathbf{V}_{t/t}\mathbf{B}_{i}^{\intercal
}(t_{n_{t}})+\mathcal{B}(t;t_{n_{t}}) \label{LLF2}%
\end{equation}
for all $t\in(t_{k},t_{k+1})$, and by
\begin{equation}
\mathbf{y}_{t_{k+1}/t_{k+1}}=\mathbf{y}_{t_{k+1}/t_{k}}+\mathbf{K}_{t_{k+1}%
}\mathbf{(\mathbf{z}}_{t_{k+1}}-\mathbf{\mathbf{C}y}_{t_{k+1}/t_{k}}\mathbf{)}
\label{LLF3}%
\end{equation}%
\begin{equation}
\mathbf{V}_{t_{k+1}/t_{k+1}}=\mathbf{V}_{t_{k+1}/t_{k}}-\mathbf{K}_{t_{k+1}%
}\mathbf{CV}_{t_{k+1}/t_{k}} \label{LLF4}%
\end{equation}
for each observation at $t_{k+1}$, with filter gain%
\begin{equation}
\mathbf{K}_{t_{k+1}}=\mathbf{V}_{t_{k+1}/t_{k}}\mathbf{C}^{\intercal
}{\Large (}\mathbf{CV}_{t_{k+1}/t_{k}}\mathbf{C}^{\intercal}+\Sigma_{t_{k+1}%
})^{-1} \label{LLF5}%
\end{equation}
for all $t_{k},t_{k+1}\in\{t\}_{M}$. The predictions
$\mathbf{y}_{t/t_{k}}$ and $\mathbf{V}_{t/t_{k}}$ are accomplished,
respectively, via expressions (\ref{LLF1})-(\ref{LLF2}) with initial
conditions $\mathbf{y}_{t_{k}/t_{k}}$ and $\mathbf{V}_{t_{k}/t_{k}}$
for $t\in(t_{k},t_{k+1}]$. Here,
\[
\mathcal{B}(t;s)=\sum\limits_{i=1}^{m}\mathbf{B}_{i}(s)\mathbf{y}%
_{t/t}\mathbf{y}_{t/t}^{\intercal}\mathbf{B}_{i}^{\intercal}(s)+\mathbf{B}%
_{i}(s)\mathbf{y}_{t/t}(\mathbf{b}_{i}^{\beta}(t;s))^{\intercal}%
+\mathbf{b}_{i}^{\beta}(t;s)\mathbf{y}_{t/t}^{\intercal}\mathbf{B}%
_{i}^{\intercal}(t)+\mathbf{b}_{i}(t;s)(\mathbf{b}_{i}^{\beta}%
(t;s))^{\intercal},
\]
and the matrices $\mathbf{A},\mathbf{B}_{i}$ and the vectors $\mathbf{a,b}%
_{i}^{\beta}$ are defined as in the WLL approximation (\ref{WLLA})
but, replacing $\mathbf{y}(s)$ by $\mathbf{y}_{s/s}$.
\end{definition}

Both, the Local Linear approximations and the Local Linearization
filters have had a number of important applications. The first ones,
in addition to the filtering problems, have been used for the
derivation of effective integration \cite{Jimenez02
SAA,Carbonell06,Carbonell08,Jimenez 2012 BIT} and inference
\cite{Shoji 1997,Shoji 1998b,Durham02,Singer02,Hurn07} methods for
SDEs, in the estimation of distribution functions in Monte Carlo
Markov Chain methods \cite{Stramer99b,Roberts01,Hansen03} and the
simulation of likelihood functions \cite{Nicolau02}. The second ones
have played a crucial role in the practical implementation of
innovation estimators for the identification of continuous-discrete
state space models \cite{Ozaki 1994,Shoji98,Ozaki JTSA,Jimenez06
JTSA}. In a variety of applications, these approximate innovation
methods have shown high effectiveness and efficiency for the
estimation of unobserved components and unknown parameters of SDEs
given a set of discrete observations. Remarkable is the
identification, from actual data, of neurophysiological, financial
and molecular models, among others (see, e.g.,
\cite{Calderon09,Kamerlin11,Chiarella09,Date11,Jimenez06
APFM,Riera04,Riera07}).

\section{Approximate Linear Minimum Variance filters}

Let $\left(  \tau\right)  _{h}$ be a time discretization of
$[t_{0},T]$ such that $\left(  \tau\right)  _{h}\supset\{t\}_{M}$,
and $\mathbf{y}_{n}$ the approximate value of $\mathbf{x}(\tau_{n})$
obtained from a discretization of the equation (\ref{SS1}) for all
$\tau_{n}\in\left(  \tau\right)  _{h}$. Let us consider the
continuous time approximation $\mathbf{y}=\{\mathbf{y}(t),$
$t\in\lbrack t_{0},T]:\mathbf{y}(\tau_{n})=\mathbf{y}_{n}$ for all
$\tau
_{n}\in\left(  \tau\right)  _{h}\}$ of $\mathbf{x}$ with initial conditions%
\[
E\left(  \mathbf{y}(t_{0})\text{{\LARGE
$\vert$%
}}\mathcal{F}_{t_{0}}\right)  =E\left(
\mathbf{x}(t_{0})\text{{\LARGE
$\vert$%
}}\mathcal{F}_{t_{0}}\right)  \text{ \ \ and \ }E\left(  \mathbf{y}%
(t_{0})\mathbf{y}^{\intercal}(t_{0})\text{{\LARGE
$\vert$%
}}\mathcal{F}_{t_{0}}\right)  =E\left(  \mathbf{x}(t_{0})\mathbf{x}%
^{\intercal}(t_{0})\text{{\LARGE
$\vert$%
}}\mathcal{F}_{t_{0}}\right)  ;\text{ }%
\]
satisfying the bound condition
\begin{equation}
E\left(  \left\vert \mathbf{y}(t)\right\vert ^{2q}\text{{\LARGE
$\vert$%
}}\mathcal{F}_{t_{k}}\right)  \leq L \label{LMVF6}%
\end{equation}
for all $t\in\lbrack t_{k},t_{k+1}]$; and the weak convergence
criteria
\begin{equation}
\underset{t_{k}\leq t\leq t_{k+1}}{\sup}\left\vert E\left(  g(\mathbf{x}%
(t))\text{{\LARGE
$\vert$%
}}\mathcal{F}_{t_{k}}\right)  -E\left( g(\mathbf{y}(t))\text{{\LARGE
$\vert$%
}}\mathcal{F}_{t_{k}}\right)  \right\vert \leq L_{k}h^{\beta} \label{LMVF7}%
\end{equation}
for all $t_{k},t_{k+1}\in \{t\}_{M}$, where $g\in\mathcal{C}_{P}^{2(\beta+1)}(\mathbb{R}%
^{d},\mathbb{R})$, $L$ and $L_{k}$ are positive constants, $%
\beta \in
\mathbb{N}
_{+}$, and $q=1,2...$.  The process $\mathbf{y}$ defined in this way
is typically called order-$\beta$ approximation to $\mathbf{x}$ in
weak sense \cite{Kloeden 1995}.

When an order-$\beta$ approximation to the solution of the state
equation (\ref{SS1}) is chosen, the following approximate filter can
be naturally defined.

\begin{definition}\label{Approx LMV filter}
Given a time discretization $\left(  \tau\right)
_{h}\supset\{t\}_{M}$, the order-$\beta$ Linear Minimum Variance
filter for the state space model (\ref{SS1})-(\ref{SS2}) is defined,
between observations, by
\begin{equation}
\mathbf{y}_{t/t}=E(\mathbf{y(}t)/Z_{t})\text{ \ \ \ \ \ \ and \ \ \ }%
\mathbf{V}_{t/t}=E(\mathbf{y(}t)\mathbf{y}^{\intercal}(t)/Z_{t})-\mathbf{y}%
_{t/t}\mathbf{y}_{t/t}^{\intercal} \label{LMVF12 App}%
\end{equation}
$\ $for all $t\in(t_{k},t_{k+1})$, and by
\begin{equation}
\mathbf{y}_{t_{k+1}/t_{k+1}}=\mathbf{y}_{t_{k+1}/t_{k}}+\mathbf{K}_{t_{k+1}%
}\mathbf{(\mathbf{z}}_{t_{k+1}}-\mathbf{\mathbf{C}y}_{t_{k+1}/t_{k}}%
\mathbf{)}, \label{LMVF3 App}%
\end{equation}%
\begin{equation}
\mathbf{V}_{t_{k+1}/t_{k+1}}=\mathbf{V}_{t_{k+1}/t_{k}}-\mathbf{K}_{t_{k+1}%
}\mathbf{CV}_{t_{k+1}/t_{k}}, \label{LMVF4 App}%
\end{equation}
for each observation at $t_{k+1}$, with filter gain%
\begin{equation}
\mathbf{K}_{t_{k+1}}=\mathbf{V}_{t_{k+1}/t_{k}}\mathbf{C}^{\intercal
}{\Large (}\mathbf{CV}_{t_{k+1}/t_{k}}\mathbf{C}^{\intercal}+\Sigma_{t_{k+1}%
})^{-1} \label{LMVF5 App}%
\end{equation}
for all $t_{k},t_{k+1}\in\{t\}_{M}$, where $\mathbf{y}$ is an
order-$\beta $ approximation to the solution of (\ref{SS1}) in weak
sense.
The predictions $\mathbf{y}_{t/t_{k}%
}=E(\mathbf{y(}t)/Z_{t_{k}})$ and $\mathbf{V}_{t/t_{k}}=E(\mathbf{y(}%
t)\mathbf{y}^{\intercal}(t)/Z_{t_{k}})-\mathbf{y}_{t/t_{k}}\mathbf{y}%
_{t/t_{k}}^{\intercal}$, with initial conditions
$\mathbf{y}_{t_{k}/t_{k}}$ and $\mathbf{V}_{t_{k}/t_{k}}$, are
defined for all $t\in(t_{k},t_{k+1}]$ and
$t_{k},t_{k+1}\in\{t\}_{M}$.
\end{definition}

Note that the goodness of the approximation $\mathbf{y}$ to
$\mathbf{x}$ is measured (in weak sense) by the left hand side of
(\ref{LMVF7}). Thus, the
inequality (\ref{LMVF7}) gives a bound for the errors of the approximation $\mathbf{y%
}$ to $\mathbf{x}$, for all $t\in \lbrack t_{k},t_{k+1}]$ and all
pair of consecutive observations $t_{k},t_{k+1}\in \{t\}_{M}$.
Moreover, this inequality states the convergence (in weak sense and
with rate $\beta $) of the approximation $\mathbf{y}$ to
$\mathbf{x}$ as the maximum stepsize $h$ of the time discretization
$(\tau )_{h}\supset \{t\}_{M}$ goes to zero. Clearly this includes,
as particular case, the convergence of the first two conditional
moments of $\mathbf{y}$ to those of $\mathbf{x}$. Since the
approximate filter in Definition \ref{Approx LMV filter} is designed
in terms of the first two conditional moments of the approximation
$\mathbf{y}$, the weak convergence of $\mathbf{y}$ to $\mathbf{x}$
should imply the convergence of the approximate filter to the exact
one. Next result deals with this matter.

\begin{theorem}
\label{Theorem CLMVF}Let $\mathbf{x}_{t/\rho }$ and
$\mathbf{U}_{t/\rho }$ be the conditional mean and variance
corresponding to the LMV filter (\ref{LMVF1})-(\ref{LMVF5}) for the
model (\ref{SS1})-(\ref{SS2}), and $\mathbf{y}_{t/\rho }$ and
$\mathbf{V}_{t/\rho }$ their respective approximations given by the
order-$\beta $ LMV filter (\ref{LMVF12 App})-(\ref{LMVF5 App}).
Then, between observations, the filters satisfy
\begin{equation}
\left\vert \mathbf{x}_{t/t}-\mathbf{y}_{t/t}\right\vert \leq K_{1}h^{\beta }%
\text{ \ \ \ \ \ and \ \ }\left\vert \mathbf{U}_{t/t}-\mathbf{V}%
_{t/t}\right\vert \leq K_{1}h^{\beta }  \label{LMVF8}
\end{equation}%
for all $t\in (t_{k},t_{k+1})$ and, at each observation $t_{k+1}$,%
\begin{equation}
\left\vert \mathbf{x}_{t_{k+1}/t_{k+1}}-\mathbf{y}_{t_{k+1}/t_{k+1}}\right%
\vert \leq K_{1}h^{\beta }\text{ \ \ \ \ and \ \ }\left\vert \mathbf{U}%
_{t_{k+1}/t_{k+1}}-\mathbf{V}_{t_{k+1}/t_{k+1}}\right\vert \leq
K_{1}h^{\beta }  \label{LMVF9}
\end{equation}%
for all $t_{k},t_{k+1}\in \{t\}_{M}$, where $K_{1}$ is a positive
constant.
For the predictions,%
\begin{equation}
\left\vert \mathbf{x}_{t/t_{k}}-\mathbf{y}_{t/t_{k}}\right\vert \leq
K_{2}h^{\beta }\text{ \ \ \ \ \ \ and \ \ \ }\left\vert \mathbf{U}_{t/t_{k}}-%
\mathbf{V}_{t/t_{k}}\right\vert \leq K_{2}h^{\beta }  \label{LMVF10}
\end{equation}%
hold for all $t\in (t_{k},t_{k+1}]$ and $t_{k},t_{k+1}\in \{t\}_{M}$, where $%
K_{2}$ is a positive constant.
\end{theorem}

\begin{proof}
Let us start proving inequalities (\ref{LMVF10}) and (\ref{LMVF8}).
For the functions $g(\mathbf{x}(t))=\mathbf{x}^{i}(t)$ and $g(%
\mathbf{x}(t))=\mathbf{x}^{i}(t)\mathbf{x}^{j}(t)$ belonging to the
function space
$\mathcal{C}_{P}^{2(\beta+1)}(\mathbb{R}^{d},\mathbb{R})$, for all
$i,j=1..d$, condition (\ref{LMVF7}) directly implies that
\begin{equation}
\left\vert \mathbf{x}_{t/t_{k}}-\mathbf{y}_{t/t_{k}}\right\vert \leq
\sqrt{d} L_{k}h^{\beta }  \label{LMVF13}
\end{equation}%
and
\begin{equation*}
\left\vert \mathbf{Q}_{t/t_{k}}-\mathbf{P}_{t/t_{k}}\right\vert \leq
d L_{k}h^{\beta }
\end{equation*}%
for all $t\in (t_{k},t_{k+1}]$, where $\mathbf{P}_{t/t_{k}}=E(\mathbf{y(}t)%
\mathbf{y}^{\intercal }(t)/Z_{t_{k}})$. Since the solution of
(\ref{SS1}) has bounded moments, there exists a positive contant
$\Lambda $ such that of $\left\vert \mathbf{x}_{t/t_{k}}\right\vert
\leq \Lambda $ for all $t\in [t_{k},t_{k+1}]$. Condition (\ref{LMVF6}) implies that $\left\vert \mathbf{y%
}_{t/t_{k}}\right\vert \leq L$ for all $t\in [t_{k},t_{k+1}]$. From
the
formula of the variance in terms of the first two moments, it follows that%
\begin{equation*}
\left\vert \mathbf{U}_{t/t_{k}}-\mathbf{V}_{t/t_{k}}\right\vert \leq
\left\vert \mathbf{Q}_{t/t_{k}}-\mathbf{P}_{t/t_{k}}\right\vert
+\left\vert
\mathbf{x}_{t/t_{k}}\mathbf{x}_{t/t_{k}}^{\intercal }-\mathbf{y}_{t/t_{k}}%
\mathbf{y}_{t/t_{k}}^{\intercal }\right\vert .
\end{equation*}%
Since
\begin{align*}
\left\vert \mathbf{x}_{t/t_{k}}\mathbf{x}_{t/t_{k}}^{\intercal }-\mathbf{y}%
_{t/t_{k}}\mathbf{y}_{t/t_{k}}^{\intercal }\right\vert & =\left\vert \mathbf{%
x}_{t/t_{k}}\mathbf{x}_{t/t_{k}}^{\intercal }-\mathbf{x}_{t/t_{k}}\mathbf{y}%
_{t/t_{k}}^{\intercal
}+\mathbf{x}_{t/t_{k}}\mathbf{y}_{t/t_{k}}^{\intercal
}-\mathbf{y}_{t/t_{k}}\mathbf{y}_{t/t_{k}}^{\intercal }\right\vert  \\
& \leq \left\vert \mathbf{x}_{t/t_{k}}(\mathbf{x}_{t/t_{k}}^{\intercal }-%
\mathbf{y}_{t/t_{k}}^{\intercal })\right\vert +\left\vert (\mathbf{x}%
_{t/t_{k}}-\mathbf{y}_{t/t_{k}})\mathbf{y}_{t/t_{k}}^{\intercal
}\right\vert
\\
& \leq (\left\vert \mathbf{x}_{t/t_{k}}\right\vert +\left\vert \mathbf{y}%
_{t/t_{k}}\right\vert )\left\vert \mathbf{x}_{t/t_{k}}-\mathbf{y}%
_{t/t_{k}}\right\vert ,
\end{align*}%
\begin{equation}
\left\vert \mathbf{U}_{t/t_{k}}-\mathbf{V}_{t/t_{k}}\right\vert \leq
\alpha _{k}h^{\beta }  \label{LMVF14}
\end{equation}%
for all $t\in (t_{k},t_{k+1}]$, where $\alpha _{k}=$
$(\sqrt{d}+L+\Lambda )\sqrt{d} L_{k}$.
Hence, inequalities (\ref{LMVF10}) are obtained from (\ref{LMVF13}) and (\ref%
{LMVF14}) with $K_{1}=\underset{k}{\max }\{\alpha _{k}\}$. Inequalities (\ref%
{LMVF8}) can be derived in the same way.

For the remainder inequalities follow this. From (\ref{LMVF3}) and (\ref%
{LMVF3 App}), it is obtained
\begin{align*}
\left\vert \mathbf{x}_{t_{k+1}/t_{k+1}}-\mathbf{y}_{t_{k+1}/t_{k+1}}\right%
\vert & \leq \left\vert \mathbf{x}_{t_{k+1}/t_{k}}-\mathbf{y}%
_{t_{k+1}/t_{k}}\right\vert  \\
& +\left\vert \mathbf{G}_{t_{k+1}}\mathbf{(\mathbf{z}}_{t_{k+1}}-\mathbf{%
\mathbf{C}x}_{t_{k+1}/t_{k}}\mathbf{)-K}_{t_{k+1}}\mathbf{(\mathbf{z}}%
_{t_{k+1}}-\mathbf{\mathbf{C}y}_{t_{k+1}/t_{k}}\mathbf{)}\right\vert  \\
& \leq (1+\left\vert
\mathbf{G}_{t_{k+1}}\mathbf{\mathbf{C}}\right\vert )\left\vert
\mathbf{x}_{t_{k+1}/t_{k}}-\mathbf{y}_{t_{k+1}/t_{k}}\right\vert
\\
& +(\left\vert \mathbf{\mathbf{z}}_{t_{k+1}}\right\vert +\left\vert \mathbf{%
\mathbf{C}y}_{t_{k+1}/t_{k}}\right\vert )\left\vert \mathbf{G}_{t_{k+1}}-%
\mathbf{K}_{t_{k+1}}\right\vert .
\end{align*}%
From (\ref{LMVF4}) and (\ref{LMVF4 App}),
\begin{align*}
\left\vert \mathbf{U}_{t_{k+1}/t_{k+1}}-\mathbf{V}_{t_{k+1}/t_{k+1}}\right%
\vert & \leq \left\vert \mathbf{U}_{t_{k+1}/t_{k}}-\mathbf{V}%
_{t_{k+1}/t_{k}}\right\vert +\left\vert \mathbf{G}_{t_{k+1}}\mathbf{CU}%
_{t_{k+1}/t_{k}}-\mathbf{K}_{t_{k+1}}\mathbf{CV}_{t_{k+1}/t_{k}}\right\vert
\\
& \leq (1+\left\vert \mathbf{G}_{t_{k+1}}\mathbf{C}\right\vert
)\left\vert
\mathbf{U}_{t_{k+1}/t_{k}}-\mathbf{V}_{t_{k+1}/t_{k}}\right\vert
+\left\vert
\mathbf{CV}_{t_{k+1}/t_{k}}\right\vert \left\vert \mathbf{G}_{t_{k+1}}-%
\mathbf{K}_{t_{k+1}}\right\vert .
\end{align*}%
By rewriting (\ref{LMVF5 App}) and (\ref{LMVF5}) as
\begin{equation*}
\mathbf{K}_{t_{k+1}}{\Large (}\mathbf{CV}_{t_{k+1}/t_{k}}\mathbf{C}%
^{\intercal }+\Sigma _{t_{k+1}})=\mathbf{V}_{t_{k+1}/t_{k}}\mathbf{C}%
^{\intercal }
\end{equation*}%
and%
\begin{equation*}
\mathbf{G}_{t_{k+1}}{\Large (}\mathbf{CU}_{t_{k+1}/t_{k}}\mathbf{C}%
^{\intercal }+\Sigma _{t_{k+1}})-\mathbf{G}_{t_{k+1}}{\Large (}\mathbf{CV}%
_{t_{k+1}/t_{k}}\mathbf{C}^{\intercal }+\Sigma _{t_{k+1}})+\mathbf{G}%
_{t_{k+1}}{\Large (}\mathbf{CV}_{t_{k+1}/t_{k}}\mathbf{C}^{\intercal
}+\Sigma _{t_{k+1}})=\mathbf{U}_{t_{k+1}/t_{k}}\mathbf{C}^{\intercal
},
\end{equation*}%
and subtracting the first expression to the second one, it follows that%
\begin{equation*}
(\mathbf{G}_{t_{k+1}}-\mathbf{K}_{t_{k+1}}){\Large (}\mathbf{CV}%
_{t_{k+1}/t_{k}}\mathbf{C}^{\intercal }+\Sigma _{t_{k+1}})=\mathbf{G}%
_{t_{k+1}}\mathbf{C(\mathbf{\mathbf{V}}}_{t_{k+1}/t_{k}}-\mathbf{U}%
_{t_{k+1}/t_{k}}\mathbf{)C}^{\intercal }+\mathbf{(U}_{t_{k+1}/t_{k}}-\mathbf{%
\mathbf{V}}_{t_{k+1}/t_{k}}\mathbf{)C}^{\intercal }.
\end{equation*}%
Thus,%
\begin{equation*}
(\mathbf{G}_{t_{k+1}}-\mathbf{K}_{t_{k+1}})=(\mathbf{I-G}_{t_{k+1}}\mathbf{%
C)(U}_{t_{k+1}/t_{k}}-\mathbf{\mathbf{V}}_{t_{k+1}/t_{k}}\mathbf{)C}%
^{\intercal }{\Large
(}\mathbf{CV}_{t_{k+1}/t_{k}}\mathbf{C}^{\intercal }+\Sigma
_{t_{k+1}})^{-1}
\end{equation*}%
and
\begin{equation*}
\left\vert \mathbf{G}_{t_{k+1}}-\mathbf{K}_{t_{k+1}}\right\vert \leq
\left\vert (\mathbf{I-G}_{t_{k+1}}\mathbf{C)}\right\vert \left\vert \mathbf{C%
}^{\intercal }{\Large
(}\mathbf{CV}_{t_{k+1}/t_{k}}\mathbf{C}^{\intercal
}+\Sigma _{t_{k+1}})^{-1}\right\vert \left\vert \mathbf{U}_{t_{k+1}/t_{k}}-%
\mathbf{\mathbf{V}}_{t_{k+1}/t_{k}}\right\vert .
\end{equation*}%
From the above inequalities, and taking into account that
$\left\vert
\mathbf{V}_{t_{k+1}/t_{k}}\right\vert $, $\left\vert \mathbf{G}%
_{t_{k+1}}\right\vert $, $\left\vert \Sigma_{t_{k+1}} \right\vert $
and $\left\vert \mathbf{C}\right\vert $ are also bound, it is
obtained that
\begin{equation*}
\left\vert \mathbf{x}_{t_{k+1}/t_{k+1}}-\mathbf{y}_{t_{k+1}/t_{k+1}}\right%
\vert \leq \beta _{k}h^{\beta }\ \text{\ \ \ \ and \ \ \ \ \ \
}\left\vert
\mathbf{U}_{t_{k+1}/t_{k+1}}-\mathbf{V}_{t_{k+1}/t_{k+1}}\right\vert
\leq \beta _{k}h^{\beta },
\end{equation*}%
where $\beta _{k}$ is a positive constant. This implies (\ref{LMVF9}) with $%
K_{2}=\underset{k}{\max }\{\beta _{k}\}$.
\end{proof}

Theorem \ref{Theorem CLMVF} states that, given a set of $M$ partial
and noisy observations of the states $\mathbf{x}$ on $\{t\}_{M}$,
the approximate LMV filter of Definition \ref{Approx LMV filter}
converges with rate $\beta $ to the exact LMV filter of Definition
\ref{LMV filter} as $h$ goes to zero, where $h$ is the maximum
stepsize of the time
discretization $(\tau )_{h}\supset \{t\}_{M}$ on which the approximation $%
\mathbf{y}$ to $\mathbf{x}$ is defined. This means that the
approximate filter inherits the convergence rate of the
approximation employed for its design. Note that, the convergence
results of Theorem \ref{Theorem CLMVF} can be easily extended for
noisy observations of any realization of $\mathbf{x}$ just by taking
expectation value in the inequalities (\ref{LMVF8})-(\ref{LMVF10}).
Further note that in both, Definition \ref{Approx LMV filter} and
Theorem \ref{Theorem CLMVF}, no restriction on the time partition
$\{t\}_{M}$ for the data has been assumed. Thus, there are not
specific constraints about the time distance between two consecutive
observations, which allows the application of the approximate filter
in a variety of practical problems (see, e.g.,
\cite{Riera07,Hu12a,Hu12b}) with a reduced number of not close
observations in time, with sequential random measurements, or with
multiple missing data.
Neither there are restrictions on the time discretization $(\tau )_{h}$ $%
\supset \{t\}_{M}$ on which the approximate filter is defined. Thus,
$(\tau )_{h}$ can be set by the user by taking into account some
specifications or previous knowledge on the filtering problem under
consideration, or automatically designed by an adaptive strategy as
it will be shown in the section concerning the numerical
simulations.

The order-$\beta $ LMV filter of Definition \ref{Approx LMV filter}
has been proposed for models with linear observation equation.
However, by following the procedure proposed in \cite{Jimenez06
JTSA}, it can be easily applied as well to models with nonlinear
observation equation.

To illustrate this, let us consider the state space model defined by
the
continuous state equation (\ref{SS1}) and the discrete observation equation%
\begin{equation}
\mathbf{z}_{t_{k}}=\mathbf{h}(t_{k},\mathbf{x}(t_{k}))+\mathbf{e}_{t_{k}%
},\text{ for }k=0,1,..,M-1, \label{SS3}%
\end{equation}
where $\mathbf{e}_{t_{k}}$ is defined as in (\ref{SS2}) and
$\mathbf{h}:$
$\mathbb{R}\times\mathbb{R}^{d}\rightarrow\mathbb{R}^{r}$ is a twice
differentiable function. By using the Ito formula,
\begin{align*}
d\mathbf{h}^{j}  &  =\mathbf{\{}\frac{\partial\mathbf{h}^{j}}{\partial t}%
+\sum\limits_{k=1}^{d}f^{k}\frac{\partial\mathbf{h}^{j}}{\partial
\mathbf{x}^{k}}+\frac{1}{2}\sum\limits_{s=1}^{m}\sum\limits_{k,l=1}%
^{d}\mathbf{g}_{s}^{l}\mathbf{g}_{s}^{k}\frac{\partial^{2}\mathbf{h}^{j}%
}{\partial\mathbf{x}^{l}\partial\mathbf{x}^{k}}\mathbf{\}}dt+\sum
\limits_{s=1}^{m}\sum\limits_{l=1}^{d}\mathbf{g}_{s}^{l}\frac{\partial
\mathbf{h}^{j}}{\partial\mathbf{x}^{l}}d\mathbf{w}^{s}\\
&  =\mathbf{\rho}^{j}dt+\sum\limits_{s=1}^{m}\mathbf{\sigma}_{s}%
^{j}d\mathbf{w}^{s}%
\end{align*}
with $j=1,..,r$. Hence, the state space model (\ref{SS1}) and
(\ref{SS3}) is transformed to the following higher-dimensional state
space model with linear observation
\[
d\mathbf{v}(t)=\mathbf{a}(t,\mathbf{v}(t))dt+\sum\limits_{i=1}^{m}%
\mathbf{b}_{i}(t,\mathbf{v}(t))d\mathbf{w}^{i}(t),
\]%
\[
\mathbf{z}_{t_{k}}=\mathbf{Cv}(t_{k})+\mathbf{e}_{t_{k}},\text{ for
}k=0,1,..,M-1,
\]
where
\[
\mathbf{v}=\left[
\begin{array}
[c]{l}%
\mathbf{x}\\
\mathbf{h}%
\end{array}
\right]  ,\text{ }\mathbf{a}=\left[
\begin{array}
[c]{l}%
\mathbf{f}\\
\mathbf{\rho}%
\end{array}
\right]  ,\text{ }\mathbf{b}_{i}=\left[
\begin{array}
[c]{l}%
\mathbf{g}_{i}\\
\mathbf{\sigma}_{i}%
\end{array}
\right]
\]
and the matrix $\mathbf{C}$ is such that $\mathbf{h}(t_{k},\mathbf{x}%
(t_{k}))=\mathbf{Cv}(t_{k})$.

In this way, the state space model (\ref{SS1}) and (\ref{SS3}) is
transformed to the form of the state space model
(\ref{SS1})-(\ref{SS2}), and so the order-$\beta $ LMV filter of
Definition \ref{Approx LMV filter} and the convergence result of
Theorem \ref{Theorem CLMVF} can be applied.

\section{Order-$\beta$ Local Linearization filters}

In principle, according to Theorem \ref{Theorem CLMVF}, any kind of
approximation $\mathbf{y}$ converging to $\mathbf{x}$ in a weak
sense can be used to construct approximate LMV filters (e.g., those
in \cite{Kloeden 1995}). Therefore, additional selection criterions
could be taking into account for this purpose. For instance, high
order of convergence, efficient algorithm for the computation of the
moments, and so on. In this paper, we elected the Local Linear
approximation (\ref{WLLA}) for the following reasons: 1) its first
two conditional moments have simple explicit formulas that can be
computed by means of efficient algorithm (including high dimensional
state equations) \cite{Jimenez 2002,Jimenez 2003,Jimenez 2012}; 2)
its first two conditional moments are exact for linear state
equations in all the possible variants (with additive and/or
multiplicative noise, autonomous or not) \cite{Jimenez 2002}; 3) it
has an adequate order of weak convergence for state equations with
additive noise \cite{Carbonell06}; and 4) the high effectiveness of
the conventional LL filters for the identification of complex
nonlinear models in a variety of applications (see, e.g.,
\cite{Calderon09,Chiarella09,Jimenez06 APFM,Riera04,Riera07}).

Once the order-$\beta$ Local Linear approximation (\ref{WLLA})\ is
chosen for approximating the state equation (\ref{SS1}), the well
know ordinary differential equations for the first two moments of
linear SDEs \cite{Arnold 1974} can be directly used to define the
following filter.

\begin{definition}\label{orderBLLfilter}
Given a time discretization $\left(  \tau\right)
_{h}\supset\{t\}_{M}$, the order-$\beta$ Local Linearization filter
for the state space model (\ref{SS1})-(\ref{SS2}) is defined,
between observations, by the piecewise linear equations
\begin{equation}
\frac{d\mathbf{y}_{t/t}}{dt}=\mathbf{A}(\tau_{n_{t}})\mathbf{y}_{t/t}%
+\mathbf{a}^{\beta}(t;\tau_{n_{t}}) \label{ALLF1}%
\end{equation}%
\begin{equation}
\frac{d\mathbf{P}_{t/t}}{dt}=\mathbf{A}(\tau_{n_{t}})\mathbf{P}_{t/t}%
+\mathbf{P}_{t/t}\mathbf{A}^{\intercal}(\tau_{n_{t}})+\sum\limits_{i=1}%
^{m}\mathbf{B}_{i}(\tau_{n_{t}})\mathbf{P}_{t/t}\mathbf{B}_{i}^{\intercal
}(\tau_{n_{t}})+\mathcal{B}(t;\tau_{n_{t}}) \label{ALLF2}%
\end{equation}%
\begin{equation}
\mathbf{V}_{t/t}=\mathbf{P}_{t/t}-\mathbf{y}_{t/t}\mathbf{y}_{t/t}^{\intercal}
\label{ALLF3}%
\end{equation}
for all $t\in(t_{k},t_{k+1})$, and by
\begin{equation}
\mathbf{y}_{t_{k+1}/t_{k+1}}=\mathbf{y}_{t_{k+1}/t_{k}}+\mathbf{K}_{t_{k+1}%
}\mathbf{(\mathbf{z}}_{t_{k+1}}-\mathbf{\mathbf{C}y}_{t_{k+1}/t_{k}}\mathbf{)}
\label{ALLF4}%
\end{equation}%
\begin{equation}
\mathbf{V}_{t_{k+1}/t_{k+1}}=\mathbf{V}_{t_{k+1}/t_{k}}-\mathbf{K}_{t_{k+1}%
}\mathbf{CV}_{t_{k+1}/t_{k}} \label{ALLF5}%
\end{equation}
for each observation at $t_{k+1}$, with filter gain%
\begin{equation}
\mathbf{K}_{t_{k+1}}=\mathbf{V}_{t_{k+1}/t_{k}}\mathbf{C}^{\intercal
}{\Large (}\mathbf{CV}_{t_{k+1}/t_{k}}\mathbf{C}^{\intercal}+\Sigma_{t_{k+1}%
})^{-1} \label{ALLF6}%
\end{equation}
for all $t_{k},t_{k+1}\in\{t\}_{M}$. Here,
\begin{align}
\mathcal{B}(t;s)  &
=\mathbf{a}^{\beta}(t;s)\mathbf{y}_{t/t}^{\intercal
}+\mathbf{y}_{t/t}(\mathbf{a}^{\beta}(t;s))^{\intercal}\nonumber\\
&  +\sum\limits_{i=1}^{m}\mathbf{B}_{i}(s)\mathbf{y}_{t/t}(\mathbf{b}%
_{i}^{\beta}(t;s))^{\intercal}+\mathbf{b}_{i}^{\beta}(t;s)\mathbf{y}%
_{t/t}^{\intercal}\mathbf{B}_{i}^{\intercal}(s)+\mathbf{b}_{i}^{\beta
}(t;s)(\mathbf{b}_{i}^{\beta}(t;s))^{\intercal} \label{ALLF7}%
\end{align}
with matrix functions $\mathbf{A},\mathbf{B}_{i}$ and vector
functions $\mathbf{a}^{\beta}\mathbf{,b}_{i}^{\beta}$ defined as in
the WLL approximation (\ref{WLLA}) but, replacing $\mathbf{y}(s)$ by
$\mathbf{y}_{s/s}$. The predictions $\mathbf{y}_{t/t_{k}}$,
$\mathbf{P}_{t/t_{k}}$ and $\mathbf{V}_{t/t_{k}}$ are accomplished,
respectively, via expressions (\ref{ALLF1})-(\ref{ALLF3}) with
initial conditions $\mathbf{y}_{t_{k}/t_{k}}$ and
$\mathbf{P}_{t_{k}/t_{k}}$ for $t\in(t_{k},t_{k+1}]$ and
$t_{k},t_{k+1}\in\{t\}_{M}$, and with
$\mathbf{A},\mathbf{B}_{i},\mathbf{a}^{\beta}\mathbf{,b}_{i}^{\beta}$
also defined as in (\ref{WLLA}) but, replacing $\mathbf{y}(s)$ by
$\mathbf{y}_{s/t_{k}}$.
\end{definition}

The approximate LL filter (\ref{ALLF1})-(\ref{ALLF6}) reduces to the
conventional LL filter (\ref{LLF1})-(\ref{LLF5}) when $\left(
\tau\right) _{h}$ $\equiv\{t\}_{M}$. For linear state equations with
multiplicative noise, the LL filter (\ref{ALLF1})-(\ref{ALLF6})
reduces to the LMV filter proposed in \cite{Jimenez 2002}, whereas
for linear state equations with additive noise, the LL filter
(\ref{ALLF1})-(\ref{ALLF6}) reduces to the classical Kalman filter.

According with Theorem \ref{Theorem CLMVF}, the approximate LL
filter (\ref{ALLF1})-(\ref{ALLF6}) will inherit the order of
convergence of the WLL approximation (\ref{WLLA}). As it was mention
before, the weak convergence rate of that approximation was early
stated in \cite{Carbonell06} for SDEs with additive noise. For
equations with multiplicative noise, this subject will be considered
in what follows.

\begin{lemma}
\label{LemmaBoundLL} Suppose that the drift and diffusion
coefficients of
the SDE (\ref{SS1}) satisfy the following conditions%
\begin{equation}
\mathbf{f}^{k},\mathbf{g}_{i}^{k}\in \mathcal{C}_{P}^{2(\beta
+1)}([a,b]\times \mathbb{R}^{d},\mathbb{R})\text{ } \label{LLF
ComponentsCond}
\end{equation}%
\begin{equation}
\left\vert \mathbf{f}(s,\mathbf{u})\right\vert
+\dsum\limits_{i=1}^{m}(\left\vert
\mathbf{g}_{i}(s,\mathbf{u})\right\vert
+\sum\limits_{k,l=1}^{d}\left\vert \mathbf{g}_{i}^{k}(s,\mathbf{u})\mathbf{g}%
_{i}^{l}(s,\mathbf{u})\right\vert \delta _{\beta }^{2})\leq
K(1+\left\vert \mathbf{u}\right\vert ),  \label{LLF
GrowthBoundLemma}
\end{equation}%
\begin{equation}
\left\vert \frac{\partial \mathbf{f}(s,\mathbf{u})}{\partial
t}\right\vert
+\left\vert \frac{\partial \mathbf{f}(s,\mathbf{u})}{\partial \mathbf{x}}%
\right\vert +\left\vert \frac{\partial ^{2}\mathbf{f}(s,\mathbf{u})}{%
\partial \mathbf{x}^{2}}\right\vert \delta _{\beta }^{2}\leq K
\label{LLF BoundLemma f}
\end{equation}%
and%
\begin{equation}
\left\vert \frac{\partial \mathbf{g}_{i}(s,\mathbf{u})}{\partial t}%
\right\vert +\left\vert \frac{\partial \mathbf{g}_{i}(s,\mathbf{u})}{%
\partial \mathbf{x}}\right\vert +\left\vert \frac{\partial ^{2}\mathbf{g}%
_{i}(s,\mathbf{u})}{\partial \mathbf{x}^{2}}\right\vert \delta
_{\beta }^{2}\leq K  \label{LLF BoundLemma g}
\end{equation}%
for all $s\in \lbrack a,b]$, $\mathbf{u}\in \mathbb{R}^{d}$, and
$i=1,..,m$,
where $K$ is a positive constant. Then the order-$\beta $ WLL approximation (%
\ref{WLLA}) satisfies%
\begin{equation}
E\left( \underset{a\leq t\leq b}{\sup }\left\vert
\mathbf{y}(t)\right\vert ^{2q}\text{{\LARGE
\TEXTsymbol{\vert}}}\mathcal{F}_{a}\right) \leq C(1+\left\vert
\mathbf{y}(a)\right\vert ^{2q})  \label{LLF GeneralBoundLL}
\end{equation}%
for each $q=1,2,\ldots ,$ where $C$ is positive constant.
\end{lemma}

\begin{proof}
Let us denote the drift and diffusion coefficients of the SDE
(\ref{WLLA}) by
\begin{equation*}
\mathbf{p}(t,\mathbf{y}(t)\mathbf{;}\tau_{n_{t}})=\mathbf{A}(\tau_{n_{t}})%
\mathbf{y}(t)+\mathbf{a}^{\mathbb{\beta}}(t;\tau_{n_{t}})
\end{equation*}
and%
\begin{equation*}
\mathbf{q}_{i}(t,\mathbf{y}(t)\mathbf{;}\tau_{n_{t}})=\mathbf{B}%
_{i}(\tau_{n_{t}})\mathbf{y}(t)+\mathbf{b}_{i}^{\mathbb{\beta}%
}(t;\tau_{n_{t}}),
\end{equation*}
respectively.

For each $q$, the Ito formula applied to $\left\vert \mathbf{y(}t\mathbf{)}%
\right\vert ^{2q}$ implies that
\begin{align*}
\left\vert \mathbf{y(}t\mathbf{)}\right\vert ^{2q} & =\left\vert \mathbf{y(}%
\tau_{n_{t}}\mathbf{)}\right\vert
^{2q}+\int\limits_{\tau_{n_{t}}}^{t}2q\left\vert
\mathbf{y}(s)\right\vert ^{2q-2}\mathbf{y}^{\intercal
}(s)\mathbf{p}(s,\mathbf{y}(s);\tau_{n_{t}})ds
\\
& +\sum\limits_{i=1}^{m}\int\limits_{\tau_{n_{t}}}^{t}2q\left\vert \mathbf{y}%
(s)\right\vert ^{2q-2}\mathbf{y}^{\intercal}(s)\mathbf{q}_{i}(s,\mathbf{y}(s)%
\mathbf{;}\tau_{n_{t}})d\mathbf{w}^{i}(s) \\
& +\sum\limits_{i=1}^{m}\int\limits_{\tau_{n_{t}}}^{t}q\left\vert \mathbf{y}%
(s)\right\vert ^{2q-2}\left\vert \mathbf{q}_{i}(s,\mathbf{y}(s)\mathbf{;}%
\tau_{n_{t}})\right\vert ^{2}ds \\
&
+\sum\limits_{i=1}^{m}\int\limits_{\tau_{n_{t}}}^{t}2q(q-1)\left\vert
\mathbf{y}(s)\right\vert ^{2q-4}\left\vert \mathbf{y}^{\intercal}(s)\mathbf{q%
}_{i}(s,\mathbf{y}(s)\mathbf{;}\tau_{n_{t}})\right\vert ^{2}ds
\end{align*}
for all $t\in\lbrack\tau_{n_{t}},\tau_{n_{t}+1}]$.

By recursive application of the expression above it is obtained that
\begin{align*}
\left\vert \mathbf{y(}t\mathbf{)}\right\vert ^{2q}& =\left\vert \mathbf{y}%
(a)\right\vert ^{2q}+\int\limits_{a}^{t}2q\left\vert \mathbf{y}%
(s)\right\vert ^{2q-2}\mathbf{y}^{\intercal }(s)\mathbf{p}(s,\mathbf{y}%
(s);\tau _{n_{s}})ds \\
& +\sum\limits_{i=1}^{m}\int\limits_{a}^{t}2q\left\vert \mathbf{y}%
(s)\right\vert ^{2q-2}\mathbf{y}^{\intercal }(s)\mathbf{q}_{i}(s,\mathbf{y}%
(s)\mathbf{;}\tau _{n_{s}})d\mathbf{w}^{i}(s) \\
& +\sum\limits_{i=1}^{m}\int\limits_{a}^{t}q\left\vert \mathbf{y}%
(s)\right\vert ^{2q-2}\left\vert \mathbf{q}_{i}(s,\mathbf{y}(s)\mathbf{;}%
\tau _{n_{s}})\right\vert ^{2}ds \\
& +\sum\limits_{i=1}^{m}\int\limits_{a}^{t}2q(q-1)\left\vert \mathbf{y}%
(s)\right\vert ^{2q-4}\left\vert \mathbf{y}^{\intercal }(s)\mathbf{q}_{i}(s,%
\mathbf{y}(s)\mathbf{;}\tau _{n_{s}})\right\vert ^{2}ds
\end{align*}%
for all $t\in \lbrack a,b]$.

Theorem 4.5.4 in \cite{Kloeden 1995} implies that $E\left(
\left\vert \mathbf{y}(t)\right\vert ^{2q}\right) <\infty $ for
$a\leq t\leq b$. Hence, the function $\mathbf{r}$ defined as
$\mathbf{r}(t)=\mathbf{0}$ for $0\leq
t<a$ and as $\mathbf{r}(t)=\left\vert \mathbf{y}(t)\right\vert ^{2q-2}%
\mathbf{y}^{\intercal
}(t)\mathbf{q}_{i}(t,\mathbf{y}(t)\mathbf{;}\tau _{n_{t}})$ for
$a\leq t\leq b$ belongs to the class $\mathcal{L}_{b}^{2}$ of
function $\mathcal{L}\times \mathcal{F}-$ measurable$.$ Then, Lemma
3.2.2 in
\cite{Kloeden 1995} implies that%
\begin{equation*}
E\left( \int\limits_{a}^{t}\left\vert \mathbf{y}(s)\right\vert ^{2q-2}%
\mathbf{y}^{\intercal
}(s)\mathbf{q}_{i}(s,\mathbf{y}(s)\mathbf{;}\tau
_{n_{s}})d\mathbf{w}^{i}(s)\right) =0
\end{equation*}%
for all $i=1,..,m$. From this and the previous expression for
$\left\vert
\mathbf{y(}t\mathbf{)}\right\vert ^{2q}$ follows that%
\begin{align*}
E\left( \underset{a\leq u\leq t}{\sup }\left\vert
\mathbf{y}(u)\right\vert ^{2q}\text{{\LARGE
\TEXTsymbol{\vert}}}\mathcal{F}_{a}\right) & \leq \left\vert
\mathbf{y}(a)\right\vert ^{2q}+2q\int\limits_{a}^{t}E\left(
\left\vert \mathbf{y}(s)\right\vert ^{2q-2}\left\vert
\mathbf{y}^{\intercal }(s)\mathbf{p}(s,\mathbf{y}(s);\tau
_{n_{s}})\right\vert \text{{\LARGE
\TEXTsymbol{\vert}}}\mathcal{F}_{a}\right) ds \\
& +q\sum\limits_{i=1}^{m}\int\limits_{a}^{t}E\left( \left\vert \mathbf{y}%
(s)\right\vert ^{2q-2}\left\vert \mathbf{q}_{i}(s,\mathbf{y}(s)\mathbf{;}%
\tau _{n_{s}})\right\vert ^{2}\text{{\LARGE \TEXTsymbol{\vert}}}\mathcal{F}%
_{a}\right) ds \\
& +2q(q-1)\sum\limits_{i=1}^{m}\int\limits_{a}^{t}E\left( \left\vert \mathbf{%
y}(s)\right\vert ^{2q-4}\left\vert \mathbf{y}^{\intercal }(s)\mathbf{q}%
_{i}(s,\mathbf{y}(s)\mathbf{;}\tau _{n_{s}})\right\vert
^{2}\text{{\LARGE \TEXTsymbol{\vert}}}\mathcal{F}_{a}\right) ds.
\end{align*}%
From conditions (\ref{LLF GrowthBoundLemma})-(\ref{LLF BoundLemma
g})
follows that%
\begin{equation*}
\left\vert \mathbf{p}(s,\mathbf{y}(s)\mathbf{;}\tau
_{n_{s}})\right\vert \leq K(\left\vert \mathbf{y}(s)\right\vert
+\left\vert \mathbf{y}(\tau _{n_{s}})\right\vert )+K_{\beta
}(1+|\mathbf{y}(s)|)+K
\end{equation*}%
and%
\begin{equation*}
\left\vert \mathbf{q}_{i}(s,\mathbf{y}(s)\mathbf{;}\tau
_{n_{s}})\right\vert \leq K(\left\vert \mathbf{y}(s)\right\vert
+\left\vert \mathbf{y}(\tau _{n_{s}})\right\vert )+K_{\beta
}(1+|\mathbf{y}(s)|)+K,
\end{equation*}%
where%
\begin{equation*}
K_{\beta }=\left\{
\begin{array}{cc}
K & \text{ for }\beta =1 \\
K(1+\frac{1}{2}K) & \text{ for }\beta =2%
\end{array}%
\right. .
\end{equation*}%
Thus, there exists a positive constant $C$ such that%
\begin{equation*}
\left\vert \mathbf{y}^{\intercal }(s)\mathbf{p}(s,\mathbf{y}(s)\mathbf{;}%
\tau _{n_{s}})\right\vert \leq C(1+|\mathbf{y(}s\mathbf{)}|^{2})+C(1+|%
\mathbf{y}(\tau _{n_{s}})|^{2}),
\end{equation*}%
\begin{equation*}
\left\vert \mathbf{q}_{i}(s,\mathbf{y}(s);\tau _{n_{s}})\right\vert
^{2}\leq C(1+|\mathbf{y(}s\mathbf{)}|^{2})+C(1+|\mathbf{y}(\tau
_{n_{s}})|^{2}),
\end{equation*}%
\begin{equation*}
\left\vert \mathbf{y}^{\intercal
}(s)\mathbf{q}_{i}(s,\mathbf{y}(s);\tau
_{n_{s}})\right\vert ^{2}\leq C|\mathbf{y(}s\mathbf{)}|^{2}(1+|\mathbf{y(}s%
\mathbf{)}|^{2})+C|\mathbf{y(}s\mathbf{)}|^{2}(1+|\mathbf{y}(\tau
_{n_{s}})|^{2}),
\end{equation*}%
and so
\begin{equation*}
E\left( \underset{a\leq u\leq t}{\sup }\left\vert
\mathbf{y}(u)\right\vert ^{2q}\text{{\LARGE
\TEXTsymbol{\vert}}}\mathcal{F}_{a}\right) \leq \left\vert
\mathbf{y}_{0}\right\vert ^{2q}+L\int\limits_{a}^{t}E\left(
\underset{a\leq u\leq s}{\sup }(1+\left\vert
\mathbf{y}(u)\right\vert
^{2})\left\vert \mathbf{y}(u)\right\vert ^{2q-2}\text{{\LARGE \TEXTsymbol{%
\vert}}}\mathcal{F}_{a}\right) ds,
\end{equation*}%
where $L=2qC(2+2qm-m)$. From the inequality $(1+z^{2})z^{2q-2}\leq 1+2z^{2q}$%
,%
\begin{equation*}
E\left( \underset{a\leq u\leq t}{\sup }\left\vert
\mathbf{y}(u)\right\vert ^{2q}\text{{\LARGE
\TEXTsymbol{\vert}}}\mathcal{F}_{a}\right) \leq \left\vert
\mathbf{y}_{0}\right\vert
^{2q}+L(t-a)+2L\int\limits_{a}^{t}E\left( \underset{a\leq u\leq s}{\sup }%
\left\vert \mathbf{y(}u)\right\vert ^{2q}\text{{\LARGE \TEXTsymbol{\vert}}}%
\mathcal{F}_{a}\right) ds.
\end{equation*}%
From this and the Gronwall Lemma, the assertion of the Theorem is
obtained.
\end{proof}

In what follows, additional notations and results of \cite{Kloeden
1995} will be used. Briefly recall us that $\mathcal{M}$ denotes the
set of all the multi-indexes $\alpha=(j_{1},\ldots,j_{l(\alpha)})$
with $j_{i}\in
\{0,1,\ldots,m\}$ and $i=1,\ldots,l(\alpha)$, where $m$ is the dimension of $%
\mathbf{w}$ in (\ref{SS1}). $l(\alpha)$ denotes the length of the
multi-index $\alpha$ and $n(\alpha)$ the number of its zero components. $%
-\alpha$ and $\alpha-$ are the multi-indexes in $\mathcal{M}$
obtained by deleting the first and the last component of $\alpha$,
respectively. The multi-index of length zero will be denoted by $v$.
Further,
\begin{equation*}
L^{0}=\frac{\partial}{\partial t}+\sum\limits_{k=1}^{d}\mathbf{f}^{k}\frac{%
\partial}{\partial\mathbf{x}^{k}}+\frac{1}{2}\sum\limits_{k,l=1}^{d}\sum%
\limits_{j=1}^{m}\mathbf{g}_{j}^{k}\mathbf{g}_{j}^{l}\text{ }\frac{%
\partial^{2}}{\partial\mathbf{x}^{k}\partial\mathbf{x}^{l}}
\end{equation*}
denotes the diffusion operator for the SDE (\ref{SS1}), and
\begin{equation*}
L^{j}=\sum\limits_{k=1}^{d}\mathbf{g}_{j}^{k}\frac{\partial}{\partial
\mathbf{x}^{k}},
\end{equation*}
for $j=1,\ldots,m$.

Let us consider the hierarchical set

\begin{equation*}
\Gamma_{\beta}=\left\{
\alpha\in\mathcal{M}:l(\alpha)\leq\beta\right\}
\end{equation*}
with $\beta=1,2$; and $\mathcal{B}(\Gamma_{\beta})=\{\mathbb{\alpha}\in%
\mathcal{M}\backslash\Gamma_{\beta}:-\mathbb{\alpha}\in\Gamma_{\beta}\}$
the remainder set of $\Gamma_{\beta}$.

\begin{lemma}
\label{LemmaLL} Let $\mathbf{y}$ be the order-$\beta $ WLL approximation (%
\ref{WLLA}), and $\mathbf{z}=\{\mathbf{z}(t),$ $t\in \lbrack a,b]\}$
be the
stochastic process defined by%
\begin{equation}
\mathbf{z}(t)=\mathbf{y}_{n_{t}}+\sum\limits_{\alpha \in \Gamma
_{\beta
}/\{\nu \}}I_{\alpha }[\Lambda _{\alpha }(\tau _{n_{t}},\mathbf{y}%
_{n_{t}};\tau _{n_{t}})]_{\tau _{n_{t}},t}+\sum\limits_{\alpha \in \mathcal{B%
}(\Gamma _{\beta })}I_{\alpha }[\Lambda _{\alpha
}(.,\mathbf{y}.;\tau _{n_{t}})]_{\tau _{n_{t}},t},
\label{WSDE-LLA-3}
\end{equation}%
where $I_{\alpha }[.]_{\tau _{n_{t}},t}$ denotes the multiple Ito
integral
and, for any given $(\tau _{n_{t}},\mathbf{y}_{n_{t}})$,%
\begin{equation*}
\Lambda _{\mathbb{\alpha }}(s,\mathbf{v};\tau _{n_{t}})=\left\{
\begin{array}{cc}
L^{j_{1}}\ldots L^{j_{l(\alpha )-1}}\mathbf{p}^{\beta
}(s,\mathbf{v};\tau
_{n_{t}}) & \text{ }if\text{ }j_{l(\alpha )}=0 \\
L^{j_{1}}\ldots L^{j_{l(\alpha )-1}}\mathbf{q}_{j_{l(\mathbb{\alpha )}%
}}^{\beta }(s,\mathbf{v};\tau _{n_{t}}) & \text{ }if\text{
}j_{l(\alpha
)}\neq 0%
\end{array}%
\right.
\end{equation*}%
is a function of $s$ and $\mathbf{v}$, with
\begin{equation*}
\mathbf{p}^{\beta }(s,\mathbf{v;}\tau _{n_{t}})=\mathbf{A}(\tau _{n_{t}})%
\mathbf{v}+\mathbf{a}^{\mathbb{\beta }}(s;\tau _{n_{t}})\text{ \ \ \
\ \ \ \
and \ \ \ \ \ \ \ }\mathbf{q}_{i}^{\beta }(s,\mathbf{v;}\tau _{n_{t}})=%
\mathbf{B}_{i}(\tau _{n_{t}})\mathbf{v}+\mathbf{b}_{i}^{\mathbb{\beta }%
}(s;\tau _{n_{t}}),
\end{equation*}%
for all $s\in \lbrack a,b]$ and $\mathbf{v}\in \mathbb{R}^{d}$, and
matrix
functions $\mathbf{A},\mathbf{B}_{i}$ and vector functions $\mathbf{a}%
^{\beta }\mathbf{,b}_{i}^{\beta }$ defined as in the WLL approximation (\ref%
{WLLA}). Then%
\begin{equation}
E\left( g(\mathbf{y}(t))\right) =E\left( g(\mathbf{z}(t))\right) ,
\label{LLF IdentityMoments}
\end{equation}%
\begin{equation}
E\left( g(\mathbf{y}(t)-\mathbf{y}(\tau _{n_{t}}))\right) =E\left( g(\mathbf{%
z}(t)-\mathbf{z}(\tau _{n_{t}}))\right)   \label{LLF IdentityMoments
b}
\end{equation}%
for all $t\in \lbrack a,b]$ and $g\in \mathcal{C}_{P}^{2(\beta +1)}(\mathbb{R%
}^{d},\mathbb{R})$; and%
\begin{equation}
I_{\alpha }[\Lambda _{\alpha }(\tau _{n_{t}},\mathbf{y}_{n_{t}};\tau
_{n_{t}})]_{\tau _{n_{t}},t}=I_{\alpha }[\lambda _{\alpha }(\tau _{n_{t}},%
\mathbf{y}_{n_{t}})]_{\tau _{n_{t}},t},  \label{LLF
IdentityIntegrals}
\end{equation}%
for all $\alpha \in \Gamma _{\beta }/\{\nu \}$ and $t\in \lbrack
a,b]$, where $\lambda _{\alpha }$ denotes the Ito coefficient
function corresponding to the SDE (\ref{SS1}).
\end{lemma}

\begin{proof}
The identities (\ref{LLF IdentityMoments})-(\ref{LLF IdentityMoments
b}) trivially hold, since (\ref{WSDE-LLA-3}) is the order-$\beta $
weak
Ito-Taylor expansion of the solution of the piecewise linear equation (\ref%
{WLLA}) with initial value $\mathbf{y}(a)=\mathbf{y}_{0}$.

By simple calculations it is obtained that Ito coefficient functions $%
\lambda_{\alpha}$ corresponding to the SDE (\ref{SS1}) are%
\begin{align*}
\lambda_{(0)}^{k} & =\mathbf{f}^{k}, \\
\lambda_{(j)}^{k} & =\mathbf{g}_{j}^{k}, \\
\lambda_{(0,j)}^{k} & =\frac{\partial\mathbf{g}_{j}^{k}}{\partial t}%
+\sum\limits_{i=1}^{d}\mathbf{f}^{i}\frac{\partial\mathbf{g}_{j}^{k}}{%
\partial\mathbf{x}^{i}}+\frac{1}{2}\sum\limits_{i,l=1}^{d}\sum%
\limits_{j=1}^{m}\mathbf{g}_{j}^{i}\mathbf{g}_{j}^{l}\text{ }\frac{%
\partial^{2}\mathbf{g}_{j}^{k}}{\partial\mathbf{x}^{i}\partial\mathbf{x}^{l}}%
, \\
\lambda_{(j,0)}^{k} & =\sum\limits_{i=1}^{d}\mathbf{g}_{j}^{i}\text{ }\frac{%
\partial\mathbf{f}^{k}}{\partial\mathbf{x}^{i}}, \\
\lambda_{(0,0)}^{k} & =\frac{\partial\mathbf{f}^{k}}{\partial t}%
+\sum\limits_{i=1}^{d}\mathbf{f}^{i}\frac{\partial\mathbf{f}^{k}}{\partial%
\mathbf{x}^{i}}+\frac{1}{2}\sum\limits_{i,l=1}^{d}\sum\limits_{j=1}^{m}%
\mathbf{g}_{j}^{i}\mathbf{g}_{j}^{l}\text{ }\frac{\partial^{2}\mathbf{f}^{k}%
}{\partial\mathbf{x}^{i}\partial\mathbf{x}^{l}}, \\
\lambda_{(i,j)}^{k} & =\sum\limits_{l=1}^{d}\mathbf{g}_{i}^{l}\frac {\partial%
\mathbf{g}_{j}^{k}}{\partial\mathbf{x}^{l}}
\end{align*}
for $\alpha\in\Gamma_{2}$. By taking into account that
$\mathbf{p}^{\beta }(s,\mathbf{v};\tau_{n})$ and
$\mathbf{q}_{i}^{\beta}(s,\mathbf{v};\tau_{n})$ are linear functions
of $s$ and $\mathbf{v}$, it is not difficult to obtain
that $\Lambda_{\mathbb{\alpha}}(\tau_{n_{s}},\mathbf{y}_{n_{s}};%
\tau_{n_{s}})=\lambda_{\mathbb{\alpha}}(\tau_{n_{s}},\mathbf{y}%
_{n_{s}})_{t_{n_{s}},s}$ for all $\alpha\in\Gamma_{2}$, which implies (\ref%
{LLF IdentityIntegrals}).
\end{proof}

Note that, the stochastic process $\mathbf{z}$ defined in the
previous lemma
is solution of the piecewise linear SDE (\ref{WLLA}) and $\Lambda _{\mathbb{%
\alpha}}$ denotes the Ito coefficient functions corresponding to
that equation. Therefore, (\ref{WSDE-LLA-3}) is the Ito-Taylor
expansion of the Local Linear approximation (\ref{WLLA}).

The main convergence result for the WLL approximations is them
stated in the following theorem.

\begin{theorem}
\label{LLF TheoremConvAppLL}Let $\mathbf{x}$ be the solution of the SDE (\ref%
{SS1}) on $[a,b]$, and $\mathbf{y}$ the order-$\beta $ weak Local
Linear approximation of $\mathbf{x}$ defined by (\ref{WLLA}).
Suppose that the drift and
diffusion coefficients of the SDE (\ref{SS1}) satisfy the conditions (\ref%
{LLF ComponentsCond})-(\ref{LLF BoundLemma g}). Further, suppose
that the initial values of $\mathbf{x}$ and $\mathbf{y}$ satisfy the
conditions
\begin{equation*}
E(\left\vert \mathbf{x}(a)\right\vert ^{q})<\infty
\end{equation*}%
and
\begin{equation*}
\left\vert E\left( g(\mathbf{x}(a))\right) -E\left(
g(\mathbf{y}(a))\right) \right\vert \leq C_{0}h^{\beta }
\end{equation*}%
for $q=1,2,\ldots ,$ some constant $C_{0}>0$ and all $g\in \mathcal{C}%
_{P}^{2(\beta +1)}(\mathbb{R}^{d},\mathbb{R})$. Then there exits a
positive
constant $C$ such that%
\begin{equation}
\left\vert E\left( g(\mathbf{x}(b))\text{{\LARGE
\TEXTsymbol{\vert}}}\mathcal{F}_{a}\right) -E\left(
g(\mathbf{y}(b))\text{{\LARGE
\TEXTsymbol{\vert}}}\mathcal{F}_{a}\right) \right\vert \leq
C(b-a)h^{\beta }. \label{GlobalOrder}
\end{equation}
\end{theorem}

\begin{proof}
For $l=1,2,\ldots,$ let $P_{l}=\{\mathbf{p}\in\{1,\ldots,d\}^{l}\}$,
and let
$\mathbf{F}_{\mathbf{p}}:\mathbb{R}^{d}\rightarrow\mathbb{R}$ be the
function defined as
\begin{equation*}
\mathbf{F}_{\mathbf{p}}(\mathbf{x})=\prod\limits_{i=1}^{l}\mathbf{x}%
^{p_{i}},
\end{equation*}
where $\mathbf{p}=(p_{1},\ldots,p_{l})\in P_{l}$.

By applying Lemma 5.11.7 in \cite{Kloeden 1995} to (\ref{WLLA}) and
taking into account that (\ref{WSDE-LLA-3}) is the order-$\beta $
weak Ito-Taylor
expansion of the solution of (\ref{WLLA}), it is obtained%
\begin{align*}
\left\vert E\left( \mathbf{F}_{\mathbf{p}}(\mathbf{y}_{n+1}-\mathbf{y}_{n})-%
\mathbf{F}_{\mathbf{p}}(\sum\limits_{\alpha \in \Gamma _{\beta
}/\{\nu \}}I_{\alpha }[\Lambda _{\alpha }(\tau
_{n},\mathbf{y}_{n};\tau _{n})]_{\tau _{n},\tau
_{n+1}})\text{{\LARGE \TEXTsymbol{\vert}}}\mathcal{F}_{\tau
_{n}}\right) \right\vert & \leq K(1+\left\vert
\mathbf{y}_{n}\right\vert
^{2r}) \\
& \cdot (\tau _{n+1}-\tau _{n})h_{n}^{\beta },
\end{align*}%
for all $\mathbf{p}\in P_{l}$ and $l=1,\ldots ,2\beta +1$, some $K>0$ and $%
r\in \{1,2,\ldots \}$, where $\Lambda _{\alpha }$ denotes the Ito
coefficient function corresponding to (\ref{WLLA}), and $h_{n}=\tau
_{n+1}-\tau _{n}$. Further, Lemma \ref{LemmaLL} implies that%
\begin{equation*}
E\left( \mathbf{F}_{\mathbf{p}}(\sum\limits_{\alpha \in \Gamma
_{\beta }/\{\nu \}}I_{\alpha }[\lambda _{\alpha }(\tau
_{n},\mathbf{y}_{n})]_{\tau _{n},\tau _{n+1}})\text{{\LARGE
\TEXTsymbol{\vert}}}\mathcal{F}_{\tau _{n}})\right) =E\left(
\mathbf{F}_{\mathbf{p}}(\sum\limits_{\alpha \in
\Gamma _{\beta }/\{\nu \}}I_{\alpha }[\Lambda _{\alpha }(\tau _{n},\mathbf{y}%
_{n};\tau _{n})]_{\tau _{n},\tau _{n+1}})\text{{\LARGE \TEXTsymbol{\vert}}}%
\mathcal{F}_{\tau _{n}}\right) ,
\end{equation*}%
where $\lambda _{\alpha }$ denotes the Ito coefficient function
corresponding to (\ref{SS1}). Hence,
\begin{align*}
\left\vert E\left( \mathbf{F}_{\mathbf{p}}(\mathbf{y}_{n+1}-\mathbf{y}_{n})-%
\mathbf{F}_{\mathbf{p}}(\sum\limits_{\alpha \in \Gamma _{\beta
}/\{\nu \}}I_{\alpha }[\lambda _{\alpha }(\tau
_{n},\mathbf{y}_{n})]_{\tau _{n},\tau _{n+1}})\text{{\LARGE
\TEXTsymbol{\vert}}}\mathcal{F}_{\tau _{n}}\right) \right\vert &
\leq K(1+\left\vert \mathbf{y}_{n}\right\vert ^{2r})(\tau
_{n+1}-\tau _{n})h_{n}^{\beta } \\
& \leq K(1+\underset{0\leq k\leq n}{\max }\left\vert \mathbf{y}%
_{k}\right\vert ^{2r}) \\
& \cdot (\tau _{n+1}-\tau _{n})h_{n}^{\beta }.
\end{align*}%
On the other hand, Theorem 4.5.4 in \cite{Kloeden 1995} applied to (\ref%
{WLLA}) and Lemma \ref{LemmaBoundLL} imply
\begin{equation*}
E\left( \left\vert \mathbf{y}_{n+1}-\mathbf{y}_{n}\right\vert ^{2q}\text{%
{\LARGE \TEXTsymbol{\vert}}}\mathcal{F}_{\tau _{n}}\right) \leq L(1+\underset%
{0\leq k\leq n}{\max }\left\vert \mathbf{y}_{k}\right\vert
^{2q})(\tau _{n+1}-\tau _{n})^{q}
\end{equation*}%
for all $0\leq n\leq N-1$, and
\begin{equation*}
E\left( \underset{0\leq k\leq n_{b}}{\max }\left\vert \mathbf{y}%
_{k}\right\vert ^{2q}\text{{\LARGE \TEXTsymbol{\vert}}}\mathcal{F}%
_{a}\right) \leq C(1+\left\vert \mathbf{y}_{0}\right\vert ^{2q}),
\end{equation*}%
respectively, where $C$ and $L$ are positive constants. The proof
concludes by using Theorem 14.5.2 in \cite{Kloeden 1995} with the
last three inequalities.
\end{proof}

For state equations with additive noise, the order of weak
convergence of the WLL approximations provided by this Theorem
matches with that early obtained in \cite{Carbonell06}.

Theorem \ref{LLF TheoremConvAppLL} provides the global order of weak
convergence for the WLL\ approximations at the time $t=b$. Notice
further that inequality (\ref{GlobalOrder}) implies that the uniform
bound
\begin{equation}
\sup_{t\in \lbrack a,b]}\left\vert E\left(
g(\mathbf{x}(t))\text{{\LARGE
\TEXTsymbol{\vert}}}\mathcal{F}_{a}\right) -E\left(
g(\mathbf{y}(t))\text{{\LARGE
\TEXTsymbol{\vert}}}\mathcal{F}_{a}\right) \right\vert \leq
C(b-a)h^{\beta } \label{UniformOrder}
\end{equation}%
holds as well for the order-$\beta $\ WLL approximation $\mathbf{y}$
since, in general, the global order of weak convergence of a
numerical
integrator implies the uniform one (see Theorem 14.5.1 and Exercise 14.5.3 in \cite%
{Kloeden 1995} for details).

Convergence in Theorem \ref{LLF TheoremConvAppLL} has been proved
under the assumption of continuity for $\mathbf{f}$ and
$\mathbf{g}_{i}$. If that is not the case, the consistency of the
WLL discretization has been proved in \cite{Stramer99}. In other
practical situations, it is important to integrate
SDEs with nonglobal Lipschitz coefficients on $%
\mathbb{R}
^{d}$ \cite{Milshtein 2005}. Typically, for such type of equations,
the conventional weak integrators display explosive values for some
realizations. In such a case, if each numerical realization of an
order-$\beta$ scheme
leaving a sufficient large sphere $\mathcal{R}\subset%
\mathbb{R}
^{d}$ is rejected, then Theorem 2.3 in \cite{Milshtein 2005} ensures
that the accuracy of the scheme is $\varepsilon+O(h^{\beta})$, where
$\varepsilon$ can be made arbitrary small with increasing the sphere
radius.\ This Theorem could be applied to the WLL approximations as
well.

Finally, the rate of convergence of the approximate Local
Linearization\ filter is states as follows.

\begin{theorem}
\label{Conv LL Filter}Given a set of $M$ partial and noisy
observations of the state equation (\ref{SS1}) on $\{t\}_{M}$, and
under the assumption that
conditions (\ref{LLF ComponentsCond})-(\ref{LLF BoundLemma g}) hold on $%
[t_{0},T]$, the approximate order-$\beta $ LL\ filter (\ref{ALLF1})-(\ref%
{ALLF6}) defined on $(\tau )_{h}\supset \{t\}_{M}$ converges with order $%
\beta $ to the exact LMV\ filter (\ref{LMVF1})-(\ref{LMVF5}) as $h$
goes to zero.
\end{theorem}

\begin{proof}
Lemma \ref{LemmaBoundLL} and Theorem \ref{LLF TheoremConvAppLL}
imply that the order-$\beta $ LL approximation $\mathbf{y}$ of
$\mathbf{x}$ defined by (\ref{WLLA}) satisfies the inequalities
(\ref{LLF GeneralBoundLL}) and (\ref{UniformOrder}) for any
integration interval $[a,b]\subset \lbrack t_{0},T]$. Thus, by
applying that lemma and theorem in each interval $[t_{k},t_{k+1}]$
with $\mathbf{y}(t_{k})\equiv \mathbf{y}_{t_{k}/t_{k}}$ (and $\mathbf{y}%
_{t_{0}/t_{0}}\equiv \mathbf{x}_{t_{0}/t_{0}}$), for
all $t_{k},t_{k+1}\in \{t\}_{M}$, the bound and convergence conditions (\ref%
{LMVF6}) and (\ref{LMVF7}) required by Theorem \ref{Theorem CLMVF}
for the convergence of the filter designed from $\mathbf{y}$ are
satisfied. Therefore, the inequalities (\ref{LMVF8})-(\ref{LMVF10})
hold for the approximate LL filter of the Definition
\ref{orderBLLfilter}, and so it has rate of convergence $\beta $
when $h$ goes to zero.
\end{proof}

\section{Practical Algorithms}

This section deals with practical implementation of the
order-$\beta$ LL\ filter (\ref{ALLF1})-(\ref{ALLF6}). Explicit
formulas for the predictions $\mathbf{y}_{t/t_{k}}$ and
$\mathbf{P}_{t/t_{k}}$, an adaptive strategy for the construction of
an adequate time discretization $\left(  \tau\right) _{h}$, and the
resulting adaptive LL filter algorithm are given.

\subsection{Formulas for the predictions\label{Formulas}}

Let us define the vectors $\mathbf{a}_{0}(\tau)$,
$\mathbf{a}_{1}(\tau)$, $\mathbf{b}_{i,0}(\tau)$ and
$\mathbf{b}_{i,1}(\tau)$ satisfying the
expressions%
\[
\mathbf{a}^{\mathbb{\beta}}(t;\tau_{n_{t}})=\mathbf{a}_{0}(\tau_{n_{t}%
})+\mathbf{a}_{1}(\tau_{n_{t}})(t-\tau_{n_{t}})\text{ \ \ \ \ and
\ \ \ \ }\mathbf{b}_{i}^{\mathbb{\beta}}(t;\tau_{n_{t}})=\mathbf{b}_{i,0}%
(\tau_{n_{t}})+\mathbf{b}_{i,1}(\tau_{n_{t}})(t-\tau_{n_{t}})
\]
for all $t\in\lbrack t_{k},t_{k+1}]$, where the vector functions
$\mathbf{a}^{\beta}$ and $\mathbf{b}_{i}^{\beta}$ are defined as in
the WLL approximation (\ref{WLLA}) but, replacing $\mathbf{y}(s)$ by
$\mathbf{y}_{s/t_{k}}.$ By simplicity, the supraindex $\beta$ is
omitted in the right hand side of the above expressions.

According Theorem 3.1 in \cite{Jimenez 2012}, the solution of the
piecewise linear differential
equations (\ref{ALLF1})-(\ref{ALLF2}) for the predictions can be computed as%
\begin{equation}
\mathbf{y}_{t/t_{k}}=\mathbf{y}_{t_{k}/t_{k}}+%
{\displaystyle\sum\limits_{n=n_{t_{k}}}^{n_{t}-1}}
\mathbf{L}_{2}e^{\mathbf{M}(\tau_{n})(\tau_{n+1}-\tau_{n})}\mathbf{u}%
_{\tau_{n},t_{k}}+\mathbf{L}_{2}e^{\mathbf{M}(\tau_{n_{t}})(t-\tau_{n_{t}}%
)}\mathbf{u}_{\tau_{n_{t}},t_{k}} \label{ALLF8}%
\end{equation}
and
\begin{equation}
vec(\mathbf{P}_{t/t_{k}})=\mathbf{L}_{1}e^{\mathbf{M}(\tau_{n_{t}})(t-\tau_{n_{t}}%
)}\mathbf{u}_{\tau_{n_{t}},t_{k}} \label{ALLF9}%
\end{equation}
for all $t\in (t_{k},t_{k+1}]$ and $t_{k},t_{k+1}\in\{t\}_{M}$,
where the vector $\mathbf{u}_{\tau,t_{k}}$ and the matrices
$\mathbf{M}(\tau)$, $\mathbf{L}_{1}$, $\mathbf{L}_{2}$\ are defined
as
\[
\mathbf{M}(\tau)=\left[
\begin{array}
[c]{cccccc}%
\mathcal{A}(\tau) & \mathcal{B}_{5}(\tau) & \mathcal{B}_{4}(\tau) &
\mathcal{B}_{3}(\tau) & \mathcal{B}_{2}(\tau) & \mathcal{B}_{1}(\tau)\\
\mathbf{0} & \mathbf{C}(\tau) & \mathbf{I}_{d+2} & \mathbf{0} &
\mathbf{0} &
\mathbf{0}\\
\mathbf{0} & \mathbf{0} & \mathbf{C}(\tau) & \mathbf{0} & \mathbf{0}
&
\mathbf{0}\\
\mathbf{0} & \mathbf{0} & \mathbf{0} & 0 & 2 & 0\\
\mathbf{0} & \mathbf{0} & \mathbf{0} & 0 & 0 & 1\\
\mathbf{0} & \mathbf{0} & \mathbf{0} & 0 & 0 & 0
\end{array}
\right]  \text{, \ \ }\mathbf{u}_{\tau,t_{k}}=\left[
\begin{array}
[c]{c}%
vec(\mathbf{P}_{\tau/t_{k}})\\
\mathbf{0}\\
\mathbf{r}\\
0\\
0\\
1
\end{array}
\right]  \in%
\mathbb{R}
^{(d^{2}+2d+7)}%
\]
and%
\[
\mathbf{L}_{1}=\left[
\begin{array}
[c]{cc}%
\mathbf{I}_{d^{2}} & \mathbf{0}_{d^{2}\times(2d+7)}%
\end{array}
\right]  \text{, \ \ \ \ \ \ \ \ \ }\mathbf{L}_{2}=\left[
\begin{array}
[c]{ccc}%
\mathbf{0}_{d\times(d^{2}+d+2)} & \mathbf{I}_{d} & \mathbf{0}_{d\times5}%
\end{array}
\right]
\]
in terms of the matrices and vectors
\[
\mathcal{A}(\tau)=\mathbf{A}(\tau)\mathbf{\oplus A}(\tau)+\sum\limits_{i=1}%
^{m}\mathbf{B}_{i}(\tau)\mathbf{\otimes B}_{i}^{\intercal}(\tau),
\]%
\[
\mathbf{C(}\tau)=\left[
\begin{array}
[c]{ccc}%
\mathbf{A}(\tau) & \mathbf{a}_{1}(\tau) & \mathbf{A}(\tau)\mathbf{y}%
_{\tau/t_{k}}+\mathbf{a}_{0}(\tau)\\
0 & 0 & 1\\
0 & 0 & 0
\end{array}
\right]  \in\mathbb{R}^{(d+2)\times(d+2)},
\]%
\[
\mathbf{r}^{\intercal}=\left[
\begin{array}
[c]{ll}%
\mathbf{0}_{1\times(d+1)} & 1
\end{array}
\right]
\]
$\mathcal{B}_{1}(\tau)=vec(\mathbf{\beta}_{1}(\tau))+\beta_{4}(\tau
)\mathbf{y}_{\tau/t_{k}}$, $\mathcal{B}_{2}(\tau)=vec(\mathbf{\beta}_{2}%
(\tau))+\mathbf{\beta}_{5}(\tau)\mathbf{y}_{\tau/t_{k}}$, $\mathcal{B}%
_{3}(\tau)=vec(\mathbf{\beta}_{3}(\tau))$, $\mathcal{B}_{4}(\tau
)=\mathbf{\beta}_{4}(\tau)\mathbf{L}$ and
$\mathcal{B}_{5}(\tau)=\mathbf{\beta }_{5}(\tau)\mathbf{L}$ with
\begin{align*}
\mathbf{\beta}_{1}(\tau)  &  =\sum\limits_{i=1}^{m}\mathbf{b}_{i,0}%
(\tau)\mathbf{b}_{i,0}^{\intercal}(\tau)\\
\mathbf{\beta}_{2}(\tau)  &  =\sum\limits_{i=1}^{m}\mathbf{b}_{i,0}%
(\tau)\mathbf{b}_{i,1}^{\intercal}(\tau)+\mathbf{b}_{i,1}(\tau)\mathbf{b}%
_{i,0}^{\intercal}(\tau)\\
\mathbf{\beta}_{3}(\tau)  &  =\sum\limits_{i=1}^{m}\mathbf{b}_{i,1}%
(\tau)\mathbf{b}_{i,1}^{\intercal}(\tau)\\
\mathbf{\beta}_{4}(\tau)  &  =\mathbf{a}_{0}(\tau)\oplus\mathbf{a}_{0}%
(\tau)+\sum\limits_{i=1}^{m}\mathbf{b}_{i,0}(\tau)\otimes\mathbf{B}_{i}%
(\tau)+\mathbf{B}_{i}(\tau)\otimes\mathbf{b}_{i,0}(\tau)\\
\mathbf{\beta}_{5}(\tau)  &  =\mathbf{a}_{1}(\tau)\oplus\mathbf{a}_{1}%
(\tau)+\sum\limits_{i=1}^{m}\mathbf{b}_{i,1}(\tau)\otimes\mathbf{B}_{i}%
(\tau)+\mathbf{B}_{i}(\tau)\otimes\mathbf{b}_{i,1}(\tau),
\end{align*}
$\mathbf{L}=\left[
\begin{array}
[c]{ll}%
\mathbf{I}_{d} & \mathbf{0}_{d\times2}%
\end{array}
\right]  $, and the $d$-dimensional identity matrix
$\mathbf{I}_{d}$. The matrix functions $\mathbf{A},\mathbf{B}_{i}$
are defined as in the WLL approximation (\ref{WLLA}) but, replacing
$\mathbf{y}(s)$ by $\mathbf{y}_{s/t_{k}}$. The symbols $vec$,
$\oplus$ and $\otimes$ denote the vectorization operator, the
Kronecker sum and product, respectively.

Alternatively, see Theorems 3.2 and 3.3 in \cite{Jimenez 2012} for
simplified formulas in the case autonomous state equations or with
additive noise.

\subsection{Adaptive selection of a time discretization\label{Adpative scheme}}

In order to write a code that automatically determines a suitable
time discretization $\left(  \tau\right)  _{h}$\ for achieving a
prescribed accuracy in the
computation of the predictions $\mathbf{y}_{t_{k+1}/t_{k}}$ and $\mathbf{P}%
_{t_{k+1}/t_{k}}$, an adequate adaptive strategy is necessary. Since
the equations (\ref{ALLF1})-(\ref{ALLF2}) for the first two
conditional moments of $\mathbf{y}$ are ordinary differential
equations, conventional adaptive strategies for numerical
integrators of such class of equations are useful. In what follows,
the adaptive strategy described in \cite{Hairer 1993} is adapted to
the LL filter requirements.

Once the values for the relative and absolute tolerances $rtol_{\mathbf{y}%
},rtol_{\mathbf{P}}$ and $atol_{\mathbf{y}},atol_{\mathbf{P}}$ for
the local errors of the first two conditional moments, for the
maximum and minimum stepsizes $h_{\max}$ and $h_{\min}$, and for the
floating point precision $prs$ are set, an initial stepsize $h_{1}$
needs to be estimated. Specifically,
\[
h_{1}=\max \{h_{\min },\min \{\delta (\mathbf{y}),\delta (vec(\mathbf{P}%
)),t_{1}-t_{0}\}\}
\]%
where
\[
\delta (\mathbf{v})=\min \{100\delta _{1}(\mathbf{v}),\delta _{2}(\mathbf{v}%
)\}
\]%
with
\[
\delta _{1}(\mathbf{v})=\left\{
\begin{array}{cc}
atol_{\mathbf{v}} & \text{if }d_{0}(\mathbf{v})<10\cdot atol_{\mathbf{v}}%
\text{ or }d_{1}(\mathbf{v})<10\cdot atol_{\mathbf{v}} \\
0.01\frac{d_{0}(\mathbf{v})}{d_{1}(\mathbf{v})} & \text{otherwise}%
\end{array}%
\right.
\]%
and%
\[
\delta _{2}(\mathbf{v})=\left\{
\begin{array}{cc}
\max \{atol_{\mathbf{v}},\delta _{1}\cdot rtol_{\mathbf{v}}\}. & \text{if }%
max\{d_{1}(\mathbf{v}),d_{2}(\mathbf{v})\}\leq prs \\
(\dfrac{0.01}{\max
\{d_{1}(\mathbf{v}),d_{2}(\mathbf{v})\}})^{\frac{1}{\beta
+1}} & \text{otherwise}%
\end{array}%
\right. .
\]%
Here, $d_{0}(\mathbf{v})=\left\Vert \mathbf{v}_{t_{0}/t_{0}}\right\Vert $, $%
d_{1}(\mathbf{v})=\left\Vert \mathbf{F}(t_{0},\mathbf{v}_{t_{0}/t_{0}})%
\right\Vert $ and $d_{2}(\mathbf{v})=\left\Vert \dfrac{\partial \mathbf{F}%
(t_{0},\mathbf{v}_{t_{0}/t_{0}})}{\partial t}+\dfrac{\partial \mathbf{F}%
(t_{0},\mathbf{v}_{t_{0}/t_{0}})}{\partial \mathbf{v}}\mathbf{F}(t_{0},%
\mathbf{v}_{t_{0}/t_{0}})\right\Vert $ are the norms of the filters
and of their first two derivatives with respect to $t$ at $t_{0}$,
where $\mathbf{F} $ is the vector field of the equation for
$\mathbf{v}$ (i.e., (\ref{ALLF1})
for $\mathbf{y}$, and (\ref{ALLF2}) for $\mathbf{P}$), and $\Vert \mathbf{v}%
\Vert =\sqrt{\dfrac{1}{\dim (\mathbf{v})}\sum_{i=1}^{\dim (\mathbf{v})}(%
\dfrac{\mathbf{v}^{i}}{\mathbf{sc}^{i}(\mathbf{v})})^{2}}$ with $\mathbf{sc}%
^{i}(\mathbf{v})=atol_{\mathbf{v}}+rtol_{\mathbf{v}}\cdot \left\vert \mathbf{%
v}_{t_{0}/t_{0}}^{i}\right\vert $.

Starting with the filter estimates $\mathbf{y}_{t_{k}/t_{k}}$ and
$\mathbf{P}_{t_{k}/t_{k}}$, the basic steps of the adaptive
algorithm for determining $\left(  \tau\right)  _{h}$ and computing
the predictions $\mathbf{y}_{t_{k+1}/t_{k}}$ and
$\mathbf{P}_{t_{k+1}/t_{k}}$ between two consecutive observations
$t_{k}$ and $t_{k+1}$ are the following:

\begin{enumerate}
\item Computation of $\mathbf{y}_{\tau_{n}/t_{k}}$ and $vec(\mathbf{P}%
_{\tau_{n}/t_{k}})$ at $\tau_{n}=\tau_{n-1}+2h_{n}$ by the recursive
evaluation of the expressions (\ref{ALLF8})-(\ref{ALLF9}) at the two
consecutive times $\tau_{n-1}+h_{n}$ and $(\tau_{n-1}+h_{n})+h_{n}$.
That is,
\[
\mathbf{y}_{\tau_{n}/t_{k}}=\mathbf{y}_{\tau_{n-1}/t_{k}}+\mathbf{L}%
_{2}e^{h_{n}\mathbf{M}(\tau_{n-1})}\mathbf{u}_{\tau_{n-1},t_{k}}%
+\mathbf{L}_{2}e^{h_{n}\mathbf{M}(\tau_{n-1}+h_{n})}\mathbf{u}_{\tau
_{n-1}+h_{n},t_{k}}%
\]
and
\[
vec(\mathbf{P}_{\tau
_{n}/t_{k}})=\mathbf{L}_{1}e^{h_{n}\mathbf{M}(\tau
_{n-1}+h_{n})}\mathbf{u}_{\tau _{n-1}+h_{n},t_{k}}.
\]%

\item Computation of an alternative estimate for the predictions at $\tau
_{n}=\tau_{n-1}+2h_{n}$ by means of the expressions%
\[
\widehat{\mathbf{y}}_{\tau_{n}/t_{k}}=\mathbf{y}_{\tau_{n-1}/t_{k}}%
+\mathbf{L}_{2}e^{2h_{n}\mathbf{M}(\tau_{n-1})}\mathbf{u}_{\tau_{n-1},t_{k}}%
\]
and
\[
vec(\widehat{\mathbf{P}}_{\tau_{n}/t_{k}})=\mathbf{L}_{1}e^{2h_{n}%
\mathbf{M}(\tau_{n-1})}\mathbf{u}_{\tau_{n-1},t_{k}},
\]
which follow from the straightforward evaluation of (\ref{ALLF8}%
)-(\ref{ALLF9}) at $\tau_{n-1}+2h_{n}$.

\item Evaluation of the error formulas
\[
E_{1}=\sqrt{\dfrac{1}{d}\sum_{i=1}^{d}(\dfrac{\mathbf{y}_{\tau_{n}/t_{k}}%
^{i}-\widehat{\mathbf{y}}_{\tau_{n}/t_{k}}^{i}}{\mathbf{sc}^{i}(\mathbf{y}%
)})^{2}}\text{ \ \ \ and \ }E_{2}=\sqrt{\dfrac{1}{d^{2}}\sum_{i=1}^{d^{2}%
}(\dfrac{\mathbf{p}_{\tau_{n}/t_{k}}^{i}-\widehat{\mathbf{p}}_{\tau_{n}/t_{k}%
}^{i}}{\mathbf{sc}^{i}(\mathbf{p})})^{2}},
\]
where $\mathbf{p}_{\tau_{n}/t_{k}}=vec(\mathbf{P}_{\tau_{n}/t_{k}})$
and
$\mathbf{sc}^{i}(\mathbf{v})=atol_{\mathbf{v}}+rtol_{\mathbf{v}}\cdot
\max \{\left\vert \mathbf{v}_{\tau _{n-1}/t_{k}}^{i}\right\vert
,\left\vert \mathbf{v}_{\tau _{n}/t_{k}}^{i}\right\vert \}.$

\item Estimation of a new stepsize%
\[
h_{new}=\max\{h_{\min},\min\{\delta_{new}(E_{1}),\delta_{new}(E_{2})\}\}
\]
where%
\[
\delta_{new}(E)=\left\{
\begin{array}
[c]{cc}%
h_{n}\cdot\min\{5,\max\{0.25,0.8\cdot(\dfrac{1}{E})^{\frac{1}{\beta+1}}\}\}
&
E\leq1\text{ }\\
h_{n}\cdot\min\{1,\max\{0.1,0.2\cdot(\dfrac{1}{E})^{\frac{1}{\beta+1}}\}\}
&
E>1\text{ }%
\end{array}
\right.
\]

\item Validation of $\mathbf{y}_{\tau_{n}/t_{k}}$ and $vec(\mathbf{P}%
_{\tau_{n}/t_{k}})$: if $max\{E_{1},E_{2}\}\leq1$ or
$h_{n}=h_{\min}$, then accept $\mathbf{y}_{\tau_{n}/t_{k}}$ and
$vec(\mathbf{P}_{\tau_{n}/t_{k}})$ as approximations to the first
two conditional moments of $\mathbf{x}$ at
$\tau_{n}=\tau_{n-1}+2h_{n}$. Otherwise, return to step 1 with $h_{n}%
=h_{new}.$

\item Control of the final stepsize: if $\tau_{n}+2h_{n}=t_{k+1}$, stop. If
$\tau_{n}+2h_{n}+h_{new}>t_{k+1}$, then redefine
$h_{new}=t_{k+1}-(\tau _{n}+2h_{n})$.

\item Return to step 1 with $n=n+1$ and $h_{n}=h_{new}$.
\end{enumerate}

Clearly, in this adaptive strategy, the selected values for the
relative and absolute tolerances will have a direct impact in the
filtering performance expressed in terms of the filtering error and
the computational time cost. Note that, under the assumed smoothness
conditions for the first two conditional moments of the state
equation, the adaptive algorithm provides an adequate estimation of
the local errors of the approximate moments at each $\tau _{n}\in
(\tau )_{h}$, and ensures that the relative and absolute errors of
the approximate moments at $\tau _{n}$ are lower than the
prearranged relative and absolute tolerance. This is done with a
computational time cost that typically increases as the values of
the tolerances decreases. Thus, for each filtering problem, adequate
tolerance values should be carefully set in advance. In practical
control engineering, these tolerances can be chosen by taking into
account the level of accuracy required by the particular problem
under consideration and the specific range of values of its state
variables.

Remarks: It is worth to emphasize that the initial stepsize $h_{1}$
is computed just one time for computing the value of
$\tau_{1}\in\lbrack t_{0},t_{1}]$. For other $\tau_{n}\in\lbrack
t_{k},t_{k+1}]$ with $n=n_{t_{k}}+1$ and $k>0$, the initial value
for the corresponding $h_{n}$ is set as $h_{n}=h_{new},$ where the
value $h_{new}$ was estimated when the previous stepsize $h_{n-1}$
was accepted. Further note that, because the flow property of the
exponential operator, only two exponential matrices need to be
evaluated in steps 1 and 2, instead of three. These two exponential
matrices can the efficiently computed through the well known
Pad\'{e} method for exponential matrices \cite{Moler03} or,
alternatively, by means of the Krylov subspace method \cite{Moler03}
in the case of high dimensional state equation. Even more, low order
Pad\'{e} and Krylov methods as suggested in \cite{Jimenez 2012 BIT}
can be used as well for reducing the computation cost, but
preserving the order-$\beta$ of the LL filters. In step 4, the
constant values in the formula for the new stepsize $\delta
_{new}(E)$ were set according to the standard integration criteria
oriented to reach an adequate balance of accuracy and computational
cost with the adaptive strategy (see, e.g., \cite{Hairer 1993}).
These values might be adjusted for improving the filtering
performance in some specific types of state equations.

\subsection{Adaptive LL filter algorithm}

Starting with the initial filter values $\mathbf{y}_{t_{0}/t_{0}}=\mathbf{x}%
_{t_{0}/t_{0}}$ and
$\mathbf{P}_{t_{0}/t_{0}}=\mathbf{Q}_{t_{0}/t_{0}}$, the adaptive LL
filter algorithm performs the recursive computation of:

\begin{enumerate}
\item the predictions $\mathbf{y}_{\tau _{n}/t_{k}}$ and $\mathbf{P}_{\tau
_{n}/t_{k}}$ for all $\tau _{n}\in \{\left( \tau \right) _{h}$ $\cap $ $%
(t_{k},t_{k+1}]\}$ by means of the recursive formulas and the
adaptive strategy of the last two subsections, and the prediction
variance by
\begin{equation*}
\mathbf{V}_{t_{k+1}/t_{k}}=\mathbf{P}_{t_{k+1}/t_{k}}-\mathbf{y}%
_{t_{k+1}/t_{k}}\mathbf{y}_{t_{k+1}/t_{k}}^{\intercal };
\end{equation*}

\item the filters
\begin{align*}
\mathbf{y}_{t_{k+1}/t_{k+1}}& =\mathbf{y}_{t_{k+1}/t_{k}}+\mathbf{K}%
_{t_{k+1}}\mathbf{(\mathbf{z}}_{t_{k+1}}-\mathbf{\mathbf{C}y}_{t_{k+1}/t_{k}}%
\mathbf{)}, \\
\mathbf{V}_{t_{k+1}/t_{k+1}}& =\mathbf{V}_{t_{k+1}/t_{k}}-\mathbf{K}%
_{t_{k+1}}\mathbf{CV}_{t_{k+1}/t_{k}}, \\
\mathbf{P}_{t_{k+1}/t_{k+1}}& =\mathbf{V}_{t_{k+1}/t_{k+1}}+\mathbf{y}%
_{t_{k+1}/t_{k+1}}\mathbf{y}_{t_{k+1}/t_{k+1}}^{\intercal },
\end{align*}%
with filter gain%
\begin{equation*}
\mathbf{K}_{t_{k+1}}=\mathbf{V}_{t_{k+1}/t_{k}}\mathbf{C}^{\intercal }%
{\Large (}\mathbf{CV}_{t_{k+1}/t_{k}}\mathbf{C}^{\intercal }+\Sigma
_{t_{k+1}})^{-1};
\end{equation*}
\end{enumerate}

\noindent for each $k$, with $k=0,1,\ldots ,M-2$. \indent

\section{Numerical Simulations}

In this section, the performance of the approximate LMV filters
introduced in this paper is illustrated, by means of simulations,
with four examples of state space models. To do so, the prediction
and filter values are computed in four different ways by means of:
1) the exact LMV filter formulas, when it is possible; 2) the
conventional LL filter; when the exact filter formulas are
available; 3) the order-$1$ LL filter with various uniform time
discretizations; and 4) the adaptive order-$1$ LL filter. For each
example, the error analysis for the estimated moments and the
estimation of the weak convergence rate are carried out through the
standard procedures (see, e.g., \cite{Kloeden 1995,Carbonell06}).

The state space models to be considered are the followings.

\textbf{Example 1.} State equation with multiplicative noise%
\begin{equation}
dx=atxdt+\sigma\sqrt{t}xdw_{1} \label{SE EJ1}%
\end{equation}
and observation equation
\begin{equation}
z_{t_{k}}=x(t_{k})+e_{t_{k}},\text{ for }k=0,1,..,M-1 \label{OE EJ1}%
\end{equation}
with $a=-0.1$, $\sigma=0.1$, $t_{0}=0.5$, $\Sigma=0.0001$,
$x_{t_{0}/t_{0}}=1$ and $Q_{t_{0}/t_{0}}=1$. For this state
equation, the predictions for the first two moments are
\[
x_{t_{k+1}/t_{k}}=x_{t_{k}/t_{k}}e^{a(t_{k+1}^{2}-t_{k}^{2})/2}\text{ \ and\ \ \ \ }%
Q_{t_{k+1}/t_{k}}=Q_{t_{k}/t_{k}}e^{(a+\sigma^{2}/2)(t_{k+1}^{2}-t_{k}^{2})},
\]
where the filters $x_{t_{k}/t_{k}}$ and $Q_{t_{k}/t_{k}}$ are
obtained from (\ref{LMVF3}) and (\ref{LMVF4}) for all
$k=0,1,..,M-2$.

\textbf{Example 2.} State equation with two additive noise%
\begin{equation}
dx=atxdt+\sigma_{1}t^{p}e^{at^{2}/2}dw_{1}+\sigma_{2}\sqrt{t}dw_{2}%
\label{SE EJ2}%
\end{equation}
and observation equation
\begin{equation}
z_{t_{k}}=x(t_{k})+e_{t_{k}},\text{ for }k=0,1,..,M-1\label{OE EJ2}%
\end{equation}
with $a=-0.25$, $p=2$, $\sigma_{1}=5,$ $\sigma_{2}=0.1$,
$t_{0}=0.01$, $\Sigma=0.0001$, $x_{t_{0}/t_{0}}=10$ and
$Q_{t_{0}/t_{0}}=100$. For this state equation, the predictions for the first two moments are%
\[
x_{t_{k+1}/t_{k}}=x_{t_{k}/t_{k}}e^{a(t_{k+1}^{2}-t_{k}^{2})/2}\text{ }%
\]
and
\[
Q_{t_{k+1}/t_{k}}=(Q_{t_{k}/t_{k}}+\frac{\sigma_{2}^{2}}{2a})e^{a(t_{k+1}^{2}%
-t_{k}^{2})}+\frac{\sigma_{1}^{2}}{2p+1}(t_{k+1}^{2p+1}-t_{k}^{2p+1})e^{at_{k+1}^{2}%
}-\frac{\sigma_{2}^{2}}{2a},
\]
where the filters $x_{t_{k}/t_{k}}$ and $Q_{t_{k}/t_{k}}$ are
obtained from (\ref{LMVF3}) and (\ref{LMVF4}) for all
$k=0,1,..,M-2$.

\textbf{Example 3.} Van der Pool oscillator with random input
\cite{Gitterman05}%
\begin{align}
dx_{1}  &  =x_{2}dt\label{SEa EJ3}\\
dx_{2}  &  =(-(x_{1}^{2}-1)x_{2}-x_{1}+a)dt+\sigma dw \label{SEb EJ3}%
\end{align}
and observation equation
\begin{equation}
z_{t_{k}}=x_{1}(t_{k})+e_{t_{k}},\text{ for }k=0,1,..,M-1, \label{OE EJ3}%
\end{equation}
where $a=0.5$ and $\sigma^{2}=(0.75)^{2}$ are the intensity and the
variance of the random input, respectively. In addition, $t_{0}=0$,
$\Sigma=0.001$, $\mathbf{x}_{t_{0}/t_{0}}^{\intercal}=[1$ $1]$ and
$\mathbf{Q}_{t_{0}/t_{0}}=\mathbf{x}_{t_{0}/t_{0}}
 \mathbf{x}_{t_{0}/t_{0}}^{\intercal}$.

\textbf{Example 4.} Van der Pool oscillator with random frequency
\cite{Gitterman05}%
\begin{align}
dx_{1}  &  =x_{2}dt\label{SEa EJ4}\\
dx_{2}  &  =(-(x_{1}^{2}-1)x_{2}-\varpi x_{1})dt+\sigma x_{1}dw
\label{SEb EJ4}%
\end{align}
and observation equation
\begin{equation}
z_{t_{k}}=x_{1}(t_{k})+e_{t_{k}},\text{ for }k=0,1,..,M-1, \label{OE EJ4}%
\end{equation}
where $\varpi=1$ and $\sigma^{2}=1$ are the frequency mean value and
variance,
respectively. In addition, $t_{0}=0$, $\Sigma=0.001$, $\mathbf{x}_{t_{0}%
/t_{0}}^{\intercal}=[1$ $1]$ and
$\mathbf{Q}_{t_{0}/t_{0}}=\mathbf{x}_{t_{0}/t_{0}}
 \mathbf{x}_{t_{0}/t_{0}}^{\intercal}$.

For each example, $2000$ realizations of the state equation solution
were computed by means of the Euler \cite{Kloeden 1995} or the Local
Linearization scheme \cite{Jimenez 2012 BIT} for the equations with
multiplicative or additive noise, respectively. For each example,
the realizations were computed over the thin time partition
$\{t_{0}+n\delta :\delta=10^{-4},n=0,..,9\times10^{4}\}$ for
guarantee a precise simulation of the stochastic solutions on the
time interval $[t_{0},t_{0}+9]$. A subsample of each realization at
the time instants $\{t\}_{M=10}=\{t_{k}=t_{0}+k:$ $k=0,..,M-1\}$ was
taken to evaluate the corresponding observation equation. In this
way, $2000$ time series $\{z_{t_{k}}^{i}\}_{k=0,..,M-1}$, with
$i=1,..2000$, of 10 values each one were finally available for every
state space example.

For each time series of the first two examples, the values of the
exact LMV filter, the conventional LL\ filter on $\{t\}_{M}$, the
order-$1$ LL filter on uniform time discretization $\left(
\tau\right) _{h}^{u}=\{\tau_{n}=t_{0}+nh:$
$n=0,..,(M-1)/h\}\supset\{t\}_{M}$ with $h=1/64,1/128,1/256,1/512$,
and the adaptive order-$1$ LL filter were computed.

For each time series $\{z_{t_{k}}^{i}\}_{k=0,..,M-1}$, four type of
errors were evaluated: the errors$\ \left\vert
\mathbf{x}_{t_{k+1}/t_{k+1}}^{i}\right.  $ $\left.
-\mathbf{y}_{t_{k+1}/t_{k+1}}^{i}\right\vert $ and $\left\vert
\mathbf{U}_{t_{k+1}/t_{k+1}}^{i}-\mathbf{V}_{t_{k+1}/t_{k+1}}^{i}\right\vert
$ between each approximate filter and the exact one, and the errors
$\left\vert
\mathbf{x}_{t_{k+1}/t_{k}}^{i}-\mathbf{y}_{t_{k+1}/t_{k}}^{i}\right\vert
$ and
$\left\vert \mathbf{U}_{t_{k+1}/t_{k}}^{i}-\mathbf{V}_{t_{k+1}/t_{k}}%
^{i}\right\vert $ between the predictions, for all $k=0,..,M-2$.%

\begin{table}[tbp] \centering
\caption{{\small Confidence limits for the errors between the exact
LMV filter
$\mathbf{x}_{t_{k+1}/t_{k+1}},\mathbf{U}_{t_{k+1}/t_{k+1}}$ of
(\ref{SE EJ1})-(\ref{OE EJ1}) and the order-$1$ LL filter
$\mathbf{y}_{t_{k+1}/t_{k+1}}^{h},\mathbf{V}_{t_{k+1}/t_{k+1}}^{h}$
on $\left(  \tau\right)  _{h}^{u}$ with different value of $h$.
Order $\widehat{\beta}$ of weak convergence estimated from the
errors.}}
$%
\resizebox{\textwidth}{!}{
\begin{tabular}
[c]{|l|l|l|l|l|l|}\hline $\mathbf{y}_{t_{k+1}/t_{k+1}}^{h}$ &
$h=1/64$ & $h=1/128$ & $h=1/256$ & $h=1/512$ &
$\widehat{\beta}$\\\hline $t_{1}/t_{1}$ & $1.36\pm0.03\times10^{-5}$
& $6.73\pm0.13\times10^{-6}$ & $3.35\pm0.06\times10^{-6}$ &
$1.67\pm0.03\times10^{-6}$ & $1.00$\\\hline $t_{2}/t_{2}$ &
$5.35\pm0.11\times10^{-6}$ & $2.66\pm0.06\times10^{-6}$ &
$1.33\pm0.03\times10^{-6}$ & $6.64\pm0.14\times10^{-7}$ &
$1.00$\\\hline $t_{3}/t_{3}$ & $3.65\pm0.06\times10^{-6}$ &
$1.82\pm0.03\times10^{-6}$ & $9.09\pm0.16\times10^{-7}$ &
$4.54\pm0.08\times10^{-7}$ & $1.00$\\\hline $t_{4}/t_{4}$ &
$3.32\pm0.10\times10^{-6}$ & $1.66\pm0.05\times10^{-6}$ &
$8.28\pm0.25\times10^{-7}$ & $4.14\pm0.12\times10^{-7}$ &
$1.00$\\\hline $t_{5}/t_{4}$ & $3.54\pm0.09\times10^{-6}$ &
$1.77\pm0.04\times10^{-6}$ & $8.82\pm0.22\times10^{-7}$ &
$4.41\pm0.11\times10^{-7}$ & $1.00$\\\hline $t_{6}/t_{6}$ &
$3.98\pm0.09\times10^{-6}$ & $1.98\pm0.05\times10^{-6}$ &
$9.91\pm0.23\times10^{-7}$ & $4.95\pm0.12\times10^{-7}$ &
$1.00$\\\hline $t_{7}/t_{7}$ & $3.42\pm0.11\times10^{-6}$ &
$1.71\pm0.05\times10^{-6}$ & $8.52\pm0.26\times10^{-7}$ &
$4.26\pm0.13\times10^{-7}$ & $1.00$\\\hline $t_{8}/t_{8}$ &
$2.00\pm0.05\times10^{-6}$ & $9.96\pm0.26\times10^{-7}$ &
$4.98\pm0.13\times10^{-7}$ & $2.49\pm0.06\times10^{-7}$ &
$1.01$\\\hline $t_{9}/t_{9}$ & $8.34\pm0.33\times10^{-7}$ &
$4.17\pm0.16\times10^{-7}$ & $2.09\pm0.08\times10^{-7}$ &
$1.05\pm0.04\times10^{-7}$ & $1.01$\\\hline
$\mathbf{V}_{t_{k+1}/t_{k+1}}^{h}$ & $h=1/64$ & $h=1/128$ &
$h=1/256$ & $h=1/512$ & $\widehat{\beta}$\\\hline $t_{1}/t_{1}$ &
$2.47\pm0.06\times10^{-5}$ & $1.23\pm0.03\times10^{-5}$ &
$6.12\pm0.14\times10^{-6}$ & $3.05\pm0.07\times10^{-6}$ &
$1.01$\\\hline $t_{2}/t_{2}$ & $8.08\pm0.20\times10^{-6}$ &
$4.03\pm0.10\times10^{-6}$ & $2.01\pm0.05\times10^{-6}$ &
$1.00\pm0.03\times10^{-6}$ & $1.00$\\\hline $t_{3}/t_{3}$ &
$4.06\pm0.11\times10^{-6}$ & $2.03\pm0.06\times10^{-6}$ &
$1.01\pm0.03\times10^{-6}$ & $5.05\pm0.14\times10^{-7}$ &
$1.00$\\\hline $t_{4}/t_{4}$ & $2.36\pm0.07\times10^{-6}$ &
$1.18\pm0.03\times10^{-6}$ & $5.87\pm0.17\times10^{-7}$ &
$2.93\pm0.08\times10^{-7}$ & $1.00$\\\hline $t_{5}/t_{4}$ &
$1.52\pm0.05\times10^{-6}$ & $7.60\pm0.24\times10^{-7}$ &
$3.78\pm0.12\times10^{-7}$ & $1.89\pm0.06\times10^{-7}$ &
$1.00$\\\hline $t_{6}/t_{6}$ & $9.36\pm0.30\times10^{-7}$ &
$4.66\pm0.15\times10^{-7}$ & $2.33\pm0.07\times10^{-7}$ &
$1.16\pm0.04\times10^{-7}$ & $1.00$\\\hline $t_{7}/t_{7}$ &
$4.49\pm0.22\times10^{-7}$ & $2.23\pm0.10\times10^{-7}$ &
$1.11\pm0.05\times10^{-7}$ & $5.55\pm0.27\times10^{-8}$ &
$1.00$\\\hline $t_{8}/t_{8}$ & $1.32\pm0.07\times10^{-7}$ &
$6.55\pm0.35\times10^{-8}$ & $3.26\pm0.17\times10^{-8}$ &
$1.63\pm0.09\times10^{-8}$ & $1.01$\\\hline $t_{9}/t_{9}$ &
$2.42\pm0.19\times10^{-8}$ & $1.19\pm0.09\times10^{-8}$ &
$5.94\pm0.46\times10^{-9}$ & $2.96\pm0.23\times10^{-9}$ &
$1.01$\\\hline
\end{tabular}}
$%
\label{Table I ALLF}%
\end{table}%

The $2000$ errors of each type were arranged into $L=20$ batches
with $K=100$ values each one, which are denoted by
$\widehat{e}_{l,j},$ $l=1,..,L;$ $j=1,...,K.$ Then, the sample mean
of the $l$-$th$ batch and of all batches can be computed by
\[
\widehat{e}_{l}=\frac{1}{K}\sum\limits_{j=1}^{K}\widehat{e}_{l,j},\text{
and }\widehat{e}=\frac{1}{L}\sum\limits_{l=1}^{L}\widehat{e}_{l},
\]
respectively. The confidence interval for each type of error is
computed as
\[
\lbrack\widehat{e}-\Delta,\widehat{e}+\Delta],
\]
where
\[
\Delta=t_{1-\alpha/2,L-1}\sqrt{\frac{\widehat{\sigma}_{e}^{2}}{L}},\text{
}\widehat{\sigma}_{e}^{2}=\frac{1}{L-1}\sum\limits_{i=1}^{L}\left\vert
\widehat{e}_{i}-\widehat{e}\right\vert ^{2},
\]
and $t_{1-\alpha/2,L-1}$ denotes the $1-\alpha/2$ percentile of the
Student's $t$ distribution with $L-1$ degrees for the significance
level $0<\alpha<1.$ The 90\% confidence interval (i.e., the values
$\Delta$ for $\alpha=0.1$) was chosen.

\bigskip%
\begin{table}[tbp] \centering
\caption{{\small Confidence limits for the errors between the exact
LMV predictions
$\mathbf{x}_{t_{k+1}/t_{k}},\mathbf{U}_{t_{k+1}/t_{k}}$of (\ref{SE
EJ1})-(\ref{OE EJ1}) and their approximations
$\mathbf{y}_{t_{k+1}/t_{k}}^{h},\mathbf{V}_{t_{k+1}/t_{k}}^{h}$
obtained by the order-$1$ LL filter on $\left(  \tau\right)
_{h}^{u}$ with different value of $h$. Order $\widehat{\beta}$ of
weak convergence estimated from the errors.} }
$%
\resizebox{\textwidth}{!}{
\begin{tabular}
[c]{|l|l|l|l|l|l|}\hline $\mathbf{y}_{t_{k+1}/t_{k}}^{h}$ & $h=1/64$
& $h=1/128$ & $h=1/256$ & $h=1/512$ & $\widehat{\beta}$\\\hline
$t_{1}/t_{0}$ & $7.35\pm0.00\times10^{-7}$ &
$1.84\pm0.00\times10^{-7}$ & $4.60\pm0.00\times10^{-8}$ &
$1.15\pm0.00\times10^{-8}$ & $2.00$\\\hline $t_{2}/t_{1}$ &
$1.11\pm0.02\times10^{-5}$ & $5.52\pm0.10\times10^{-6}$ &
$2.74\pm0.05\times10^{-6}$ & $1.37\pm0.03\times10^{-6}$ &
$1.01$\\\hline $t_{3}/t_{2}$ & $4.22\pm0.09\times10^{-6}$ &
$2.02\pm0.04\times10^{-6}$ & $9.95\pm0.21\times10^{-7}$ &
$4.94\pm0.10\times10^{-7}$ & $1.03$\\\hline $t_{4}/t_{3}$ &
$2.75\pm0.05\times10^{-6}$ & $1.27\pm0.02\times10^{-6}$ &
$6.20\pm0.11\times10^{-7}$ & $3.07\pm0.05\times10^{-7}$ &
$1.05$\\\hline $t_{5}/t_{4}$ & $2.20\pm0.06\times10^{-6}$ &
$1.03\pm0.03\times10^{-6}$ & $5.07\pm0.15\times10^{-7}$ &
$2.52\pm0.08\times10^{-7}$ & $1.04$\\\hline $t_{6}/t_{5}$ &
$2.06\pm0.05\times10^{-6}$ & $9.88\pm0.26\times10^{-7}$ &
$4.88\pm0.12\times10^{-7}$ & $2.43\pm0.06\times10^{-7}$ &
$1.03$\\\hline $t_{7}/t_{6}$ & $2.02\pm0.05\times10^{-6}$ &
$9.93\pm0.24\times10^{-7}$ & $4.94\pm0.12\times10^{-7}$ &
$2.46\pm0.06\times10^{-7}$ & $1.01$\\\hline $t_{8}/t_{7}$ &
$1.57\pm0.05\times10^{-6}$ & $7.74\pm0.24\times10^{-7}$ &
$3.84\pm0.12\times10^{-7}$ & $1.92\pm0.06\times10^{-7}$ &
$1.01$\\\hline $t_{9}/t_{8}$ & $8.18\pm0.22\times10^{-7}$ &
$4.06\pm0.11\times10^{-7}$ & $2.03\pm0.05\times10^{-7}$ &
$1.01\pm0.03\times10^{-7}$ & $1.00$\\\hline
$\mathbf{V}_{t_{k+1}/t_{k}}^{h}$ & $h=1/64$ & $h=1/128$ & $h=1/256$
& $h=1/512$ & $\widehat{\beta}$\\\hline $t_{1}/t_{0}$ &
$1.22\pm0.00\times10^{-4}$ & $6.14\pm0.00\times10^{-5}$ &
$3.08\pm0.00\times10^{-5}$ & $1.54\pm0.00\times10^{-5}$ &
$1.01$\\\hline $t_{2}/t_{1}$ & $7.61\pm0.02\times10^{-5}$ &
$3.85\pm0.00\times10^{-5}$ & $1.94\pm0.00\times10^{-5}$ &
$9.71\pm0.02\times10^{-6}$ & $1.00$\\\hline $t_{3}/t_{2}$ &
$4.05\pm0.04\times10^{-5}$ & $2.06\pm0.02\times10^{-5}$ &
$1.04\pm0.00\times10^{-5}$ & $5.22\pm0.05\times10^{-6}$ &
$1.00$\\\hline $t_{4}/t_{3}$ & $1.77\pm0.04\times10^{-5}$ &
$9.06\pm0.17\times10^{-6}$ & $4.59\pm0.09\times10^{-6}$ &
$2.31\pm0.04\times10^{-6}$ & $1.00$\\\hline $t_{5}/t_{4}$ &
$6.10\pm0.14\times10^{-6}$ & $3.16\pm0.07\times10^{-6}$ &
$1.61\pm0.04\times10^{-6}$ & $8.09\pm0.18\times10^{-7}$ &
$1.00$\\\hline $t_{6}/t_{5}$ & $1.68\pm0.05\times10^{-6}$ &
$8.81\pm0.23\times10^{-7}$ & $4.51\pm0.12\times10^{-7}$ &
$2.28\pm0.06\times10^{-7}$ & $1.00$\\\hline $t_{7}/t_{6}$ &
$3.79\pm0.11\times10^{-7}$ & $1.99\pm0.06\times10^{-7}$ &
$1.02\pm0.03\times10^{-7}$ & $5.17\pm0.15\times10^{-8}$ &
$1.00$\\\hline $t_{8}/t_{7}$ & $8.01\pm0.34\times10^{-8}$ &
$4.14\pm0.18\times10^{-8}$ & $2.10\pm0.09\times10^{-8}$ &
$1.06\pm0.04\times10^{-8}$ & $1.00$\\\hline $t_{9}/t_{8}$ &
$1.58\pm0.07\times10^{-8}$ & $7.83\pm0.35\times10^{-9}$ &
$3.90\pm0.17\times10^{-9}$ & $1.95\pm0.09\times10^{-9}$ &
$1.00$\\\hline
\end{tabular}}
$%
\label{Table II ALLF}%
\end{table}%

\subsection{Results for Example 1}

Tables \ref{Table I ALLF}-\ref{Table III ALLF} show the estimated
errors for the state space model (\ref{SE EJ1})-(\ref{OE EJ1}).
Specifically, Table \ref{Table I ALLF} shows the confidence limits
for the errors between the exact LMV filter
$\mathbf{x}_{t_{k+1}/t_{k+1}},\mathbf{U}_{t_{k+1}/t_{k+1}}$
and the order-$1$ LL filter $\mathbf{y}_{t_{k+1}/t_{k+1}},\mathbf{V}%
_{t_{k+1}/t_{k+1}}$ on the time discretization $\left(  \tau\right)
_{h}^{u}$, with $h=1/64,1/128,1/256,1/512$. Table \ref{Table II
ALLF} shows the confidence limits for the errors between the exact
LMV predictions
$\mathbf{x}_{t_{k+1}/t_{k}},\mathbf{U}_{t_{k+1}/t_{k}}$ and their
approximations
$\mathbf{y}_{t_{k+1}/t_{k}},\mathbf{V}_{t_{k+1}/t_{k}}$ obtained by
the order-$1$ LL filter on $\left(  \tau\right)  _{h}^{u}$. Table
\ref{Table III ALLF} shows the confidence limits for the errors
between the moments of the exact LMV filter and their respective
approximations obtained by the conventional LL filter and the
adaptive LL filter. The average of accepted and fail steps of the
adaptive LL filter at each $t_{k}\in\{t\}_{M}$ is given in Figure
\ref{Fig1}. The absolute and relative tolerances for the first and
second moments were set as
$rtol_{\mathbf{y}}=rtol_{\mathbf{P}}=5\times
10^{-9}$ and $atol_{\mathbf{y}}=5\times10^{-9}$, $atol_{\mathbf{P}}%
=5\times10^{-12}$. Note as the accuracy of the LL filter on uniform
discretizations $\left(  \tau\right)  _{h}^{u}$ improve as $h$
decreases, and the large difference among the accuracy of the
conventional and the adaptive LL filter.

For each approximate conditional moment, the estimated order
$\widehat{\beta}$ of weak convergence were obtained as the slope of
the straight line fitted to
the set of four points $\{\log_{2}(h_{j})$, $\log_{2}(\widehat{e}(h_{j}%
))\}_{j=1,..,4}$ taken from their corresponding errors tables
\ref{Table I ALLF} and \ref{Table II ALLF}. The values
$\widehat{\beta}$ are shown in these tables as well. The estimates
$\widehat{\beta}\approx1$ corroborate the theoretical value for
$\beta$ given in Theorem \ref{Conv LL Filter}. The estimate
$\widehat{\beta}=2.00$ corresponding to
$\mathbf{y}_{t_{1}/t_{0}}^{h}$ in Table \ref{Table II ALLF} agrees
with the
expected estimate of $\beta$ for the equation (\ref{SE EJ1}) on $[t_{0,}%
t_{1}]$. In this particular situation, the exact prediction $\mathbf{x}%
_{t_{1}/t_{0}}$ given by (\ref{LMVF1}) reduces to an ordinary
differential equation and the LL prediction formula (\ref{ALLF1})
reduces to the classical order-$2$ LL integrator for such class of
equations (see, e.g., \cite{Jimenez02 AMC}). In the others
subintervals $[t_{k,}t_{k+1}]$ with $k\neq0$, the prediction
$\mathbf{y}_{t_{k+1}/t_{k}}^{h}$ depends nonlinearly
of $\mathbf{y}$ through the initial value $\mathbf{y}_{t_{k+1}/t_{k+1}}^{h}$.%

\begin{table}[tbp] \centering
\caption{{\small Confidence limits for the errors between the exact
LMV filter and predictions of (\ref{SE EJ1})-(\ref{OE EJ1}) with
their corresponding approximations obtained by the conventional LL
filter and the adaptive LL filter, which are denoted with
superscripts $0$ and $A$, respectively.}}
$%
\begin{tabular}
[c]{|l|l|l|l|l|}\hline $k$ & $\mathbf{y}_{t_{k+1}/t_{k}}^{0}$ &
$\mathbf{y}_{t_{k+1}/t_{k}}^{A}$ & $\mathbf{V}_{t_{k+1}/t_{k}}^{0}$
& $\mathbf{V}_{t_{k+1}/t_{k}}^{A}$\\\hline $0$ &
$2.79\pm0.00\times10^{-3}$ & $5.09\pm0.00\times10^{-10}$ &
$1.75\pm0.00\times10^{-3}$ & $3.23\pm0.00\times10^{-6}$\\\hline $1$
& $5.62\pm0.13\times10^{-3}$ & $2.86\pm0.05\times10^{-7}$ & $5.42\pm
0.15\times10^{-3}$ & $2.09\pm0.00\times10^{-6}$\\\hline $2$ &
$6.09\pm0.05\times10^{-3}$ & $1.06\pm0.02\times10^{-8}$ & $4.04\pm
0.07\times10^{-3}$ & $1.16\pm0.01\times10^{-6}$\\\hline $3$ &
$5.74\pm0.06\times10^{-3}$ & $6.75\pm0.12\times10^{-8}$ & $3.16\pm
0.07\times10^{-3}$ & $5.29\pm0.10\times10^{-7}$\\\hline $4$ &
$4.54\pm0.05\times10^{-3}$ & $5.73\pm0.17\times10^{-8}$ & $1.70\pm
0.04\times10^{-3}$ & $1.92\pm0.04\times10^{-7}$\\\hline $5$ &
$3.17\pm0.04\times10^{-3}$ & $5.72\pm0.15\times10^{-8}$ & $7.21\pm
0.20\times10^{-4}$ & $5.62\pm0.15\times10^{-8}$\\\hline $6$ &
$2.01\pm0.03\times10^{-3}$ & $6.07\pm0.14\times10^{-8}$ & $2.44\pm
0.07\times10^{-4}$ & $1.34\pm0.04\times10^{-8}$\\\hline $7$ &
$1.24\pm0.02\times10^{-3}$ & $5.02\pm0.15\times10^{-8}$ & $7.37\pm
0.27\times10^{-5}$ & $2.85\pm0.12\times10^{-9}$\\\hline $8$ &
$7.32\pm0.14\times10^{-4}$ & $2.82\pm0.07\times10^{-8}$ & $1.81\pm
0.09\times10^{-5}$ & $5.39\pm0.23\times10^{-10}$\\\hline $k$ &
$\mathbf{y}_{t_{k+1}/t_{k+1}}^{0}$ &
$\mathbf{y}_{t_{k+1}/t_{k+1}}^{A}$
& $\mathbf{V}_{t_{k+1}/t_{k+1}}^{0}$ & $\mathbf{V}_{t_{k+1}/t_{k+1}}^{A}%
$\\\hline $0$ & $3.94\pm0.08\times10^{-3}$ &
$3.50\pm0.07\times10^{-7}$ & $7.22\pm 0.17\times10^{-3}$ &
$6.38\pm0.14\times10^{-7}$\\\hline $1$ & $6.25\pm0.13\times10^{-4}$
& $1.43\pm0.03\times10^{-7}$ & $9.61\pm 0.25\times10^{-4}$ &
$2.16\pm0.05\times10^{-7}$\\\hline $2$ & $3.58\pm0.07\times10^{-4}$
& $1.01\pm0.02\times10^{-7}$ & $4.12\pm 0.12\times10^{-4}$ &
$1.12\pm0.03\times10^{-7}$\\\hline $3$ & $3.09\pm0.09\times10^{-4}$
& $9.44\pm0.29\times10^{-8}$ & $2.33\pm 0.07\times10^{-4}$ &
$6.69\pm0.19\times10^{-8}$\\\hline $4$ & $3.50\pm0.09\times10^{-4}$
& $1.04\pm0.03\times10^{-7}$ & $1.64\pm 0.06\times10^{-4}$ &
$4.45\pm0.14\times10^{-8}$\\\hline $5$ & $4.49\pm0.12\times10^{-4}$
& $1.22\pm0.03\times10^{-7}$ & $1.16\pm 0.04\times10^{-4}$ &
$2.86\pm0.09\times10^{-8}$\\\hline $6$ & $5.93\pm0.10\times10^{-4}$
& $1.12\pm0.03\times10^{-7}$ & $7.96\pm 0.28\times10^{-5}$ &
$1.45\pm0.07\times10^{-8}$\\\hline $7$ & $6.61\pm0.12\times10^{-4}$
& $6.93\pm0.18\times10^{-8}$ & $3.92\pm 0.17\times10^{-5}$ &
$4.50\pm0.23\times10^{-9}$\\\hline $8$ & $5.91\pm0.09\times10^{-4}$
& $2.89\pm0.11\times10^{-8}$ & $1.48\pm 0.06\times10^{-5}$ &
$8.17\pm0.63\times10^{-10}$\\\hline
\end{tabular}
$%
\label{Table III ALLF}%
\end{table}%

\subsection{Results for Example 2}

Tables \ref{Table IV ALLF}-\ref{Table VI ALLF} show the estimated
errors for the state space model (\ref{SE EJ2})-(\ref{OE EJ2}). In
particular, Table \ref{Table IV ALLF} shows the confidence limits
for the errors between the exact LMV filter
$\mathbf{x}_{t_{k+1}/t_{k+1}},\mathbf{U}_{t_{k+1}/t_{k+1}}$
and the order-$1$ LL filter $\mathbf{y}_{t_{k+1}/t_{k+1}},\mathbf{V}%
_{t_{k+1}/t_{k+1}}$ on the time discretization $\left(  \tau\right)
_{h}^{u}$, with $h=1/64,1/128,1/256,1/512$. Table \ref{Table V ALLF}
shows the confidence
limits for the errors between the exact LMV predictions $\mathbf{x}%
_{t_{k+1}/t_{k}}$, $\mathbf{U}_{t_{k+1}/t_{k}}$ and their
approximations $\mathbf{y}_{t_{k+1}/t_{k}}$,
$\mathbf{V}_{t_{k+1}/t_{k}}$ obtained by the order-$1$ LL filter on
$\left(  \tau\right)  _{h}^{u}$. Table \ref{Table VI ALLF} shows the
confidence limits for the errors between the moments of the exact
LMV filter and their respective approximations obtained by the
conventional LL filter and the adaptive LL filter. The average of
accepted and fail steps of the adaptive LL filter at each
$t_{k}\in\{t\}_{M}$ is given in Figure \ref{Fig1}. The absolute and
relative tolerances for the first and
second moments for this filter were set as $rtol_{\mathbf{y}}=rtol_{\mathbf{P}%
}=5\times10^{-8}$ and $atol_{\mathbf{y}}=5\times10^{-8}$, $atol_{\mathbf{P}%
}=5\times10^{-11}$. Note as the accuracy of the LL filter on uniform
discretizations $\left(  \tau\right)  _{h}^{u}$ improve as $h$
decreases, and the large difference among the accuracy of the
conventional and the adaptive LL filter.

For each approximate conditional moment, the estimated order
$\widehat{\beta}$ of weak convergence were obtained as the slope of
the straight line fitted to the set of four points $\left\{
\log_{2}(h_{j}),\right.  $ \  $\left.  \log
_{2}(\widehat{e}(h_{j}))\right\}  _{j=1,..,4}$ taken from their
corresponding errors tables \ref{Table IV ALLF} and \ref{Table V
ALLF}. The values $\widehat{\beta}$ are included in these tables
too. The estimates $\widehat{\beta}\approx1$ corroborate the
theoretical value for $\beta$ given in Theorem \ref{Conv LL Filter}.
The estimate $\widehat{\beta}\approx2.00$ corresponding to
$\mathbf{y}_{t_{k+1}/t_{k}}^{h}$ in Table \ref{Table V ALLF} agrees
with the expected estimate of $\beta$ for the equation (\ref{SE
EJ2}) on $[t_{k,}t_{k+1}]$, for all $k$. Similarly to the previous
example, the exact prediction $\mathbf{x}_{t_{k+1}/t_{k}}$ given by
(\ref{LMVF1}) reduces to an ordinary differential equation and the
LL prediction formula (\ref{ALLF1}) reduces as well to the classical
order-$2$ LL integrator for all
$k$. Contrary to the first example, in this one, the prediction $\mathbf{y}%
_{t_{k+1}/t_{k}}^{h}$ with $k\neq0$ does not depend of $\mathbf{y}$
through the initial value $\mathbf{y}_{t_{k+1}/t_{k+1}}^{h}$ and so
the estimate $\widehat{\beta}\approx2.00$ is preserved.

\bigskip\bigskip%
\begin{table}[tbp] \centering
\caption{{\small Confidence limits for the errors between the exact
LMV filter
$\mathbf{x}_{t_{k+1}/t_{k+1}},\mathbf{U}_{t_{k+1}/t_{k+1}}$ of
(\ref{SE EJ2})-(\ref{OE EJ2}) and the order-$1$ LL filter
$\mathbf{y}_{t_{k+1}/t_{k+1}}^{h},\mathbf{V}_{t_{k+1}/t_{k+1}}^{h}$
on $\left(  \tau\right)  _{h}^{u}$ with different value of $h$.
Order $\widehat{\beta}$ of weak convergence estimated from the
errors.} }
$%
\resizebox{\textwidth}{!}{
\begin{tabular}
[c]{|l|l|l|l|l|l|}\hline $\mathbf{y}_{t_{k+1}/t_{k+1}}^{h}$ &
$h=1/64$ & $h=1/128$ & $h=1/256$ & $h=1/512$ &
$\widehat{\beta}$\\\hline $t_{1}/t_{1}$ & $2.00\pm0.04\times10^{-8}$
& $1.17\pm0.02\times10^{-8}$ & $6.23\pm0.11\times10^{-9}$ &
$3.22\pm0.06\times10^{-9}$ & $0.95$\\\hline $t_{2}/t_{2}$ &
$1.31\pm0.03\times10^{-8}$ & $6.44\pm0.14\times10^{-8}$ &
$3.20\pm0.07\times10^{-9}$ & $1.59\pm0.04\times10^{-9}$ &
$1.02$\\\hline $t_{3}/t_{3}$ & $1.12\pm0.03\times10^{-8}$ &
$5.52\pm0.14\times10^{-8}$ & $2.74\pm0.06\times10^{-9}$ &
$1.36\pm0.03\times10^{-9}$ & $1.02$\\\hline $t_{4}/t_{4}$ &
$1.56\pm0.03\times10^{-8}$ & $7.74\pm0.12\times10^{-8}$ &
$3.85\pm0.06\times10^{-9}$ & $1.92\pm0.03\times10^{-9}$ &
$1.01$\\\hline $t_{5}/t_{4}$ & $2.95\pm0.07\times10^{-8}$ &
$1.47\pm0.03\times10^{-8}$ & $7.38\pm0.17\times10^{-9}$ &
$3.69\pm0.08\times10^{-9}$ & $1.01$\\\hline $t_{6}/t_{6}$ &
$7.85\pm0.19\times10^{-8}$ & $3.98\pm0.09\times10^{-8}$ &
$2.01\pm0.05\times10^{-8}$ & $1.01\pm0.02\times10^{-8}$ &
$0.99$\\\hline $t_{7}/t_{7}$ & $2.65\pm0.06\times10^{-7}$ &
$1.37\pm0.03\times10^{-7}$ & $6.94\pm0.15\times10^{-8}$ &
$3.50\pm0.07\times10^{-8}$ & $0.99$\\\hline $t_{8}/t_{8}$ &
$5.46\pm0.16\times107$ & $2.79\pm0.08\times10^{-7}$ &
$1.41\pm0.04\times10^{-7}$ & $7.09\pm0.21\times10^{-8}$ &
$0.99$\\\hline $t_{9}/t_{9}$ & $4.76\pm0.13\times10^{-7}$ &
$2.37\pm0.06\times10^{-7}$ & $1.18\pm0.03\times10^{-7}$ &
$5.91\pm0.16\times10^{-8}$ & $1.01$\\\hline
$\mathbf{V}_{t_{k+1}/t_{k+1}}^{h}$ & $h=1/64$ & $h=1/128$ &
$h=1/256$ & $h=1/512$ & $\widehat{\beta}$\\\hline $t_{1}/t_{1}$ &
$3.48\pm0.09\times10^{-7}$ & $2.03\pm0.05\times10^{-7}$ &
$1.09\pm0.03\times10^{-7}$ & $5.60\pm0.14\times10^{-8}$ &
$0.88$\\\hline $t_{2}/t_{2}$ & $2.66\pm0.11\times10^{-7}$ &
$1.31\pm0.05\times10^{-7}$ & $6.51\pm0.26\times10^{-8}$ &
$3.24\pm0.13\times10^{-8}$ & $1.01$\\\hline $t_{3}/t_{3}$ &
$2.97\pm0.12\times10^{-7}$ & $1.46\pm0.06\times10^{-7}$ &
$7.24\pm0.30\times10^{-8}$ & $3.61\pm0.15\times10^{-8}$ &
$1.01$\\\hline $t_{4}/t_{4}$ & $3.46\pm0.11\times10^{-7}$ &
$1.71\pm0.05\times10^{-7}$ & $8.53\pm0.27\times10^{-8}$ &
$4.26\pm0.13\times10^{-8}$ & $1.01$\\\hline $t_{5}/t_{4}$ &
$3.44\pm0.16\times10^{-7}$ & $1.73\pm0.08\times10^{-7}$ &
$8.65\pm0.41\times10^{-8}$ & $4.33\pm0.21\times10^{-8}$ &
$1.01$\\\hline $t_{6}/t_{6}$ & $3.58\pm0.15\times10^{-7}$ &
$1.83\pm0.07\times10^{-7}$ & $9.21\pm0.38\times10^{-8}$ &
$4.63\pm0.19\times10^{-8}$ & $0.98$\\\hline $t_{7}/t_{7}$ &
$3.57\pm0.14\times10^{-7}$ & $1.85\pm0.07\times10^{-7}$ &
$9.42\pm0.38\times10^{-8}$ & $4.75\pm0.19\times10^{-8}$ &
$0.97$\\\hline $t_{8}/t_{8}$ & $2.35\pm0.13\times10^{-7}$ &
$1.21\pm0.07\times10^{-7}$ & $6.11\pm0.34\times10^{-8}$ &
$3.08\pm0.17\times10^{-8}$ & $0.98$\\\hline $t_{9}/t_{9}$ &
$1.67\pm0.09\times10^{-7}$ & $8.31\pm0.04\times10^{-8}$ &
$4.15\pm0.22\times10^{-8}$ & $2.07\pm0.11\times10^{-8}$ &
$1.00$\\\hline
\end{tabular}}
$%
\label{Table IV ALLF}%
\end{table}%
%

\begin{table}[tbp] \centering
\caption{{\small Confidence limits for the errors between the exact
LMV predictions
$\mathbf{x}_{t_{k+1}/t_{k}},\mathbf{U}_{t_{k+1}/t_{k}}$of (\ref{SE
EJ2})-(\ref{OE EJ2}) and their approximations
$\mathbf{y}_{t_{k+1}/t_{k}}^{h},\mathbf{V}_{t_{k+1}/t_{k}}^{h}$
obtained by the order-$1$ LL filter on $\left(  \tau\right)
_{h}^{u}$ with different value of $h$. Order $\widehat{\beta}$ of
weak convergence estimated from the errors.} }
$%
\resizebox{\textwidth}{!}{
\begin{tabular}
[c]{|l|l|l|l|l|l|}\hline $\mathbf{y}_{t_{k+1}/t_{k}}^{h}$ & $h=1/64$
& $h=1/128$ & $h=1/256$ & $h=1/512$ & $\widehat{\beta}$\\\hline
$t_{1}/t_{0}$ & $2.28\pm0.00\times10^{-5}$ &
$5.70\pm0.00\times10^{-6}$ & $1.43\pm0.00\times10^{-6}$ &
$3.57\pm0.00\times10^{-7}$ & $2.00$\\\hline $t_{2}/t_{1}$ &
$4.63\pm0.03\times10^{-5}$ & $1.16\pm0.00\times10^{-5}$ &
$2.89\pm0.02\times10^{-6}$ & $7.22\pm0.05\times10^{-7}$ &
$2.00$\\\hline $t_{3}/t_{2}$ & $5.44\pm0.11\times10^{-5}$ &
$1.36\pm0.03\times10^{-5}$ & $3.39\pm0.07\times10^{-6}$ &
$8.47\pm0.17\times10^{-7}$ & $2.00$\\\hline $t_{4}/t_{3}$ &
$6.91\pm0.13\times10^{-5}$ & $1.72\pm0.03\times10^{-5}$ &
$4.30\pm0.08\times10^{-6}$ & $1.07\pm0.02\times10^{-6}$ &
$2.00$\\\hline $t_{5}/t_{4}$ & $5.85\pm0.12\times10^{-5}$ &
$1.46\pm0.03\times10^{-5}$ & $3.64\pm0.08\times10^{-6}$ &
$9.09\pm0.19\times10^{-7}$ & $2.00$\\\hline $t_{6}/t_{5}$ &
$3.10\pm0.08\times10^{-5}$ & $7.73\pm0.21\times10^{-6}$ &
$1.93\pm0.05\times10^{-6}$ & $4.81\pm0.13\times10^{-7}$ &
$2.00$\\\hline $t_{7}/t_{6}$ & $1.13\pm0.03\times10^{-5}$ &
$2.82\pm0.06\times10^{-6}$ & $7.01\pm0.16\times10^{-7}$ &
$1.74\pm0.04\times10^{-7}$ & $2.01$\\\hline $t_{8}/t_{7}$ &
$3.07\pm0.07\times10^{-6}$ & $7.56\pm0.18\times10^{-7}$ &
$1.84\pm0.04\times10^{-7}$ & $4.40\pm0.11\times10^{-8}$ &
$2.04$\\\hline $t_{9}/t_{8}$ & $7.63\pm0.24\times10^{-7}$ &
$1.75\pm0.05\times10^{-7}$ & $3.73\pm0.11\times10^{-8}$ &
$6.97\pm0.16\times10^{-9}$ & $2.25$\\\hline
$\mathbf{V}_{t_{k+1}/t_{k}}^{h}$ & $h=1/64$ & $h=1/128$ & $h=1/256$
& $h=1/512$ & $\widehat{\beta}$\\\hline $t_{1}/t_{0}$ &
$2.43\pm0.00\times10^{-3}$ & $1.28\pm0.00\times10^{-3}$ &
$6.56\pm0.00\times10^{-4}$ & $3.32\pm0.00\times10^{-4}$ &
$0.88$\\\hline $t_{2}/t_{1}$ & $7.22\pm0.00\times10^{-2}$ &
$3.54\pm0.00\times10^{-2}$ & $1.75\pm0.00\times10^{-2}$ &
$8.73\pm0.00\times10^{-3}$ & $1.01$\\\hline $t_{3}/t_{2}$ &
$1.69\pm0.00\times10^{-1}$ & $8.29\pm0.00\times10^{-2}$ &
$4.11\pm0.00\times10^{-2}$ & $2.04\pm0.00\times10^{-2}$ &
$1.01$\\\hline $t_{4}/t_{3}$ & $1.16\pm0.00\times10^{-1}$ &
$5.73\pm0.00\times10^{-2}$ & $2.84\pm0.00\times10^{-2}$ &
$1.42\pm0.00\times10^{-2}$ & $1.01$\\\hline $t_{5}/t_{4}$ &
$3.38\pm0.00\times10^{-2}$ & $1.68\pm0.00\times10^{-2}$ &
$8.36\pm0.00\times10^{-3}$ & $4.17\pm0.00\times10^{-3}$ &
$1.00$\\\hline $t_{6}/t_{5}$ & $4.81\pm0.00\times10^{-3}$ &
$2.41\pm0.00\times10^{-3}$ & $1.21\pm0.00\times10^{-3}$ &
$6.05\pm0.00\times10^{-4}$ & $0.99$\\\hline $t_{7}/t_{6}$ &
$3.77\pm0.00\times10^{-4}$ & $1.91\pm0.00\times10^{-4}$ &
$9.62\pm0.00\times10^{-5}$ & $4.83\pm0.00\times10^{-5}$ &
$0.97$\\\hline $t_{8}/t_{7}$ & $3.27\pm0.00\times10^{-5}$ &
$1.65\pm0.00\times10^{-5}$ & $8.28\pm0.00\times10^{-6}$ &
$4.15\pm0.00\times10^{-6}$ & $0.98$\\\hline $t_{9}/t_{8}$ &
$1.70\pm0.00\times10^{-5}$ & $8.44\pm0.00\times10^{-6}$ &
$4.21\pm0.00\times10^{-6}$ & $2.10\pm0.00\times10^{-6}$ &
$1.00$\\\hline
\end{tabular}}
$%
\label{Table V ALLF}%
\end{table}%
%

\begin{table}[tbp] \centering
\caption{{\small Confidence limits for the errors between the exact
LMV filter and predictions of (\ref{SE EJ2})-(\ref{OE EJ2}) with
their corresponding approximations obtained by the conventional LL
filter and the adaptive LL filter, which are denoted with
superscripts $0$ and $A$, respectively.}}
$%
\begin{tabular}
[c]{|l|l|l|l|l|}\hline $k$ & $\mathbf{y}_{t_{k+1}/t_{k}}^{0}$ &
$\mathbf{y}_{t_{k+1}/t_{k}}^{A}$ & $\mathbf{V}_{t_{k+1}/t_{k}}^{0}$
& $\mathbf{V}_{t_{k+1}/t_{k}}^{A}$\\\hline $0$ &
$7.69\pm0.00\times10^{-2}$ & $2.17\pm0.00\times10^{-6}$ &
$2.63\pm0.00$ & $3.72\pm0.00\times10^{-4}$\\\hline $1$ &
$2.09\pm0.01\times10^{-1}$ & $2.14\pm0.04\times10^{-7}$ &
$8.01\pm0.03$ & $1.85\pm0.00\times10^{-3}$\\\hline $2$ &
$2.81\pm0.06\times10^{-1}$ & $8.41\pm0.38\times10^{-8}$ & $4.93\pm
0.13\times10^{2}$ & $3.24\pm0.02\times10^{-3}$\\\hline $3$ &
$4.02\pm0.07\times10^{-1}$ & $1.26\pm0.07\times10^{-7}$ & $3.22\pm
0.17\times10^{2}$ & $2.33\pm0.02\times10^{-3}$\\\hline $4$ &
$3.82\pm0.08\times10^{-1}$ & $1.55\pm0.08\times10^{-7}$ & $6.18\pm
0.10\times10^{1}$ & $7.45\pm0.06\times10^{-4}$\\\hline $5$ &
$2.27\pm0.06\times10^{-1}$ & $1.06\pm0.06\times10^{-7}$ & $6.23\pm
0.25\times10^{-1}$ & $1.15\pm0.01\times10^{-4}$\\\hline $6$ &
$9.36\pm0.21\times10^{-2}$ & $4.68\pm0.22\times10^{-8}$ & $8.34\pm
0.25\times10^{-2}$ & $9.74\pm0.11\times10^{-6}$\\\hline $7$ &
$2.89\pm0.07\times10^{-2}$ & $1.23\pm0.08\times10^{-8}$ & $8.19\pm
0.28\times10^{-3}$ & $7.43\pm0.08\times10^{-7}$\\\hline $8$ &
$8.63\pm0.28\times10^{-3}$ & $1.10\pm0.06\times10^{-9}$ & $2.73\pm
0.02\times10^{-3}$ & $2.86\pm0.00\times10^{-7}$\\\hline $k$ &
$\mathbf{y}_{t_{k+1}/t_{k+1}}^{0}$ &
$\mathbf{y}_{t_{k+1}/t_{k+1}}^{A}$
& $\mathbf{V}_{t_{k+1}/t_{k+1}}^{0}$ & $\mathbf{V}_{t_{k+1}/t_{k+1}}^{A}%
$\\\hline $0$ & $1.75\pm0.03\times10^{-3}$ &
$3.29\pm0.06\times10^{-9}$ & $3.08\pm 0.08\times10^{-2}$ &
$5.73\pm0.14\times10^{-8}$\\\hline $1$ & $9.47\pm0.24\times10^{-7}$
& $3.37\pm0.07\times10^{-10}$ & $1.62\pm0.07\times10^{-5}$ &
$6.84\pm0.26\times10^{-9}$\\\hline $2$ & $2.17\pm0.06\times10^{-6}$
& $2.16\pm0.05\times10^{-10}$ & $5.58\pm0.23\times10^{-5}$ &
$5.75\pm0.25\times10^{-9}$\\\hline $3$ & $2.73\pm0.04\times10^{-6}$
& $3.15\pm0.05\times10^{-10}$ & $5.70\pm0.17\times10^{-5}$ &
$7.10\pm0.26\times10^{-9}$\\\hline $4$ & $3.02\pm0.05\times10^{-6}$
& $6.58\pm0.14\times10^{-10}$ & $2.79\pm0.11\times10^{-5}$ &
$8.01\pm0.42\times10^{-9}$\\\hline $5$ & $1.35\pm0.03\times10^{-5}$
& $1.94\pm0.06\times10^{-9}$ & $7.37\pm 0.28\times10^{-5}$ &
$9.32\pm0.44\times10^{-9}$\\\hline $6$ & $1.16\pm0.14\times10^{-3}$
& $7.07\pm0.20\times10^{-9}$ & $1.89\pm 0.09\times10^{-4}$ &
$1.04\pm0.06\times10^{-8}$\\\hline $7$ & $1.01\pm0.03\times10^{-3}$
& $1.26\pm0.05\times10^{-8}$ & $3.18\pm 0.16\times10^{-5}$ &
$5.56\pm0.36\times10^{-9}$\\\hline $8$ & $6.91\pm0.18\times10^{-5}$
& $8.05\pm0.21\times10^{-9}$ & $2.01\pm 0.12\times10^{-5}$ &
$2.83\pm0.16\times10^{-9}$\\\hline
\end{tabular}
$%
\label{Table VI ALLF}%
\end{table}%

\subsection{Results for Examples 3 and 4}%

\begin{table}[tbp] \centering
\caption{{\small Estimate order of convergence $\widehat{\beta}$ for
the moments of the order-$1$ LL filter applied to the state space
models (\ref{SEa EJ3})-(\ref{OE EJ3}) and (\ref{SEa EJ4})-(\ref{OE
EJ4}) corresponding to the Van der Pool oscillator with additive
(Add) and multiplicative (Mul) noise, respectively.} }
$%
\resizebox{\textwidth}{!}{
\begin{tabular}
[c]{|l|l|l|l|l|l|l|l|l|l|}\hline
$k\backslash Add$ & $\mathbf{y}_{t_{k+1}/t_{k}}$ & $\mathbf{V}_{t_{k+1}/t_{k}%
}$ & $\mathbf{y}_{t_{k+1}/t_{k+1}}$ & $\mathbf{V}_{t_{k+1}/t_{k+1}}$
&
$k\backslash Mul$ & $\mathbf{y}_{t_{k+1}/t_{k}}$ & $\mathbf{V}_{t_{k+1}/t_{k}%
}$ & $\mathbf{y}_{t_{k+1}/t_{k+1}}$ &
$\mathbf{V}_{t_{k+1}/t_{k+1}}$\\\hline $0$ & $1.10$ & $1.04$ &
$1.11$ & $1.04$ & $0$ & $1.08$ & $1.01$ & $1.03$ & $1.01$\\\hline
$1$ & $1.04$ & $1.05$ & $1.04$ & $1.05$ & $1$ & $1.03$ & $1.03$ &
$1.04$ & $1.03$\\\hline $2$ & $1.03$ & $1.03$ & $1.03$ & $1.03$ &
$2$ & $1.03$ & $1.04$ & $1.05$ & $1.04$\\\hline $3$ & $1.02$ &
$1.02$ & $1.03$ & $1.02$ & $3$ & $1.03$ & $1.03$ & $1.04$ &
$1.02$\\\hline $4$ & $1.01$ & $1.01$ & $1.01$ & $0.97$ & $4$ &
$1.02$ & $1.02$ & $1.02$ & $1.02$\\\hline $5$ & $1.01$ & $1.03$ &
$1.01$ & $1.01$ & $5$ & $1.02$ & $0.94$ & $1.02$ & $0.83$\\\hline
$6$ & $1.02$ & $1.01$ & $1.01$ & $0.98$ & $6$ & $1.01$ & $0.98$ &
$1.01$ & $0.97$\\\hline $7$ & $1.02$ & $1.04$ & $1.02$ & $1.06$ &
$7$ & $0.97$ & $1.00$ & $1.02$ & $0.99$\\\hline $8$ & $1.03$ &
$1.02$ & $1.02$ & $1.02$ & $8$ & $1.02$ & $1.01$ & $1.03$ &
$0.99$\\\hline
\end{tabular}}
$%
\label{Table VII ALLF}%
\end{table}%
\normalsize

Since explicit formulas of the LMV filter for the state space models
(\ref{SEa EJ3})-(\ref{OE EJ3}) and (\ref{SEa EJ4})-(\ref{OE EJ4})
are not available, the error analysis of the previous examples
should be adjusted. In this situation, by taking into account the
results of the previous examples, the moments estimated by the
adaptive LL filter with small tolerance can be used as a precise
estimation for the moments of the exact LMV filter. By doing this,
the confidence interval for the errors can similarly be computed as
before for estimate the order $\widehat{\beta}$ of weak convergence
of the order-$1$ LL filter. Table \ref{Table VII ALLF} shows the
estimated order $\widehat {\beta}$ of weak convergence obtained, as
explained above, as the slope of the straight line fitted to the set
of four points $\left\{  \log_{2}(h_{j}),\right.  $ \  $\left.
\log_{2}(\widehat{e}(h_{j}))\right\}  _{j=1,..,4}$, where
$\widehat{e}(h_{j})$ denotes the error between the order-$1$ LL
filter on $\left( \tau\right)  _{h_{j}}^{u}$, with
$h_{j}=1/2^{5+j}$, and the adaptive LL filter with small tolerance.
The tolerances for the adaptive filter were set as
$rtol_{\mathbf{y}}=rtol_{\mathbf{P}}=5\times10^{-8}$ and $atol_{\mathbf{y}%
}=5\times10^{-8}$, $atol_{\mathbf{P}}=5\times10^{-11}$ in the model
(\ref{SEa EJ3})-(\ref{OE EJ3}), and as $rtol_{\mathbf{y}}=rtol_{\mathbf{P}%
}=10^{-7}$ and $atol_{\mathbf{y}}=10^{-7}$, $atol_{\mathbf{P}%
}=10^{-10}$ in the model (\ref{SEa EJ4})-(\ref{OE EJ4}). For each
model, the average of accepted and fail steps of the adaptive LL
filter at each $t_{k}\in\{t\}_{M}$ is given in Figure \ref{Fig1}.
Notice that, for both examples, the estimates
$\widehat{\beta}\approx1$ corroborate the theoretical value for
$\beta$ stated in Theorem \ref{Conv LL Filter}.

\begin{figure}[h]
\centering $%
\begin{array}
[c]{c}%
\includegraphics[width=3.2in]{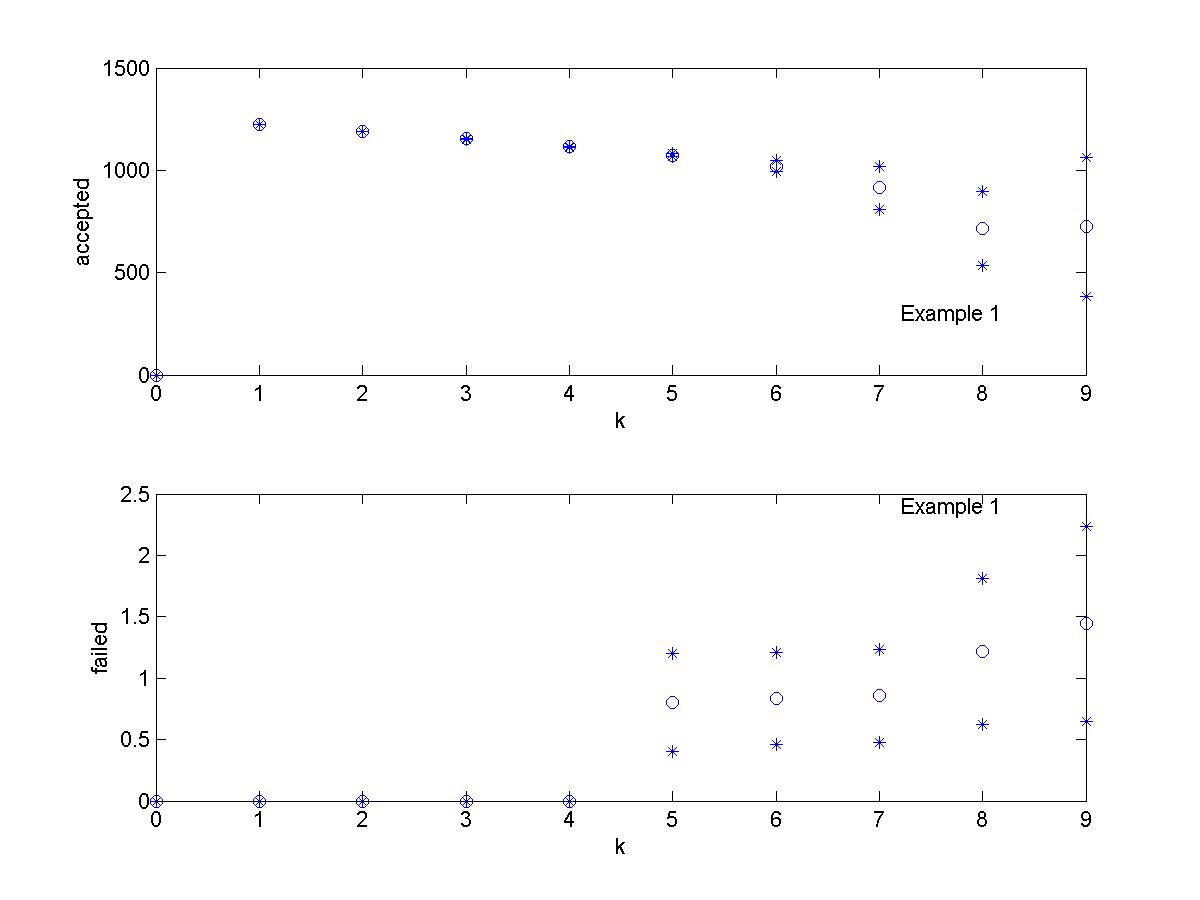}
\includegraphics[width=3.2in]{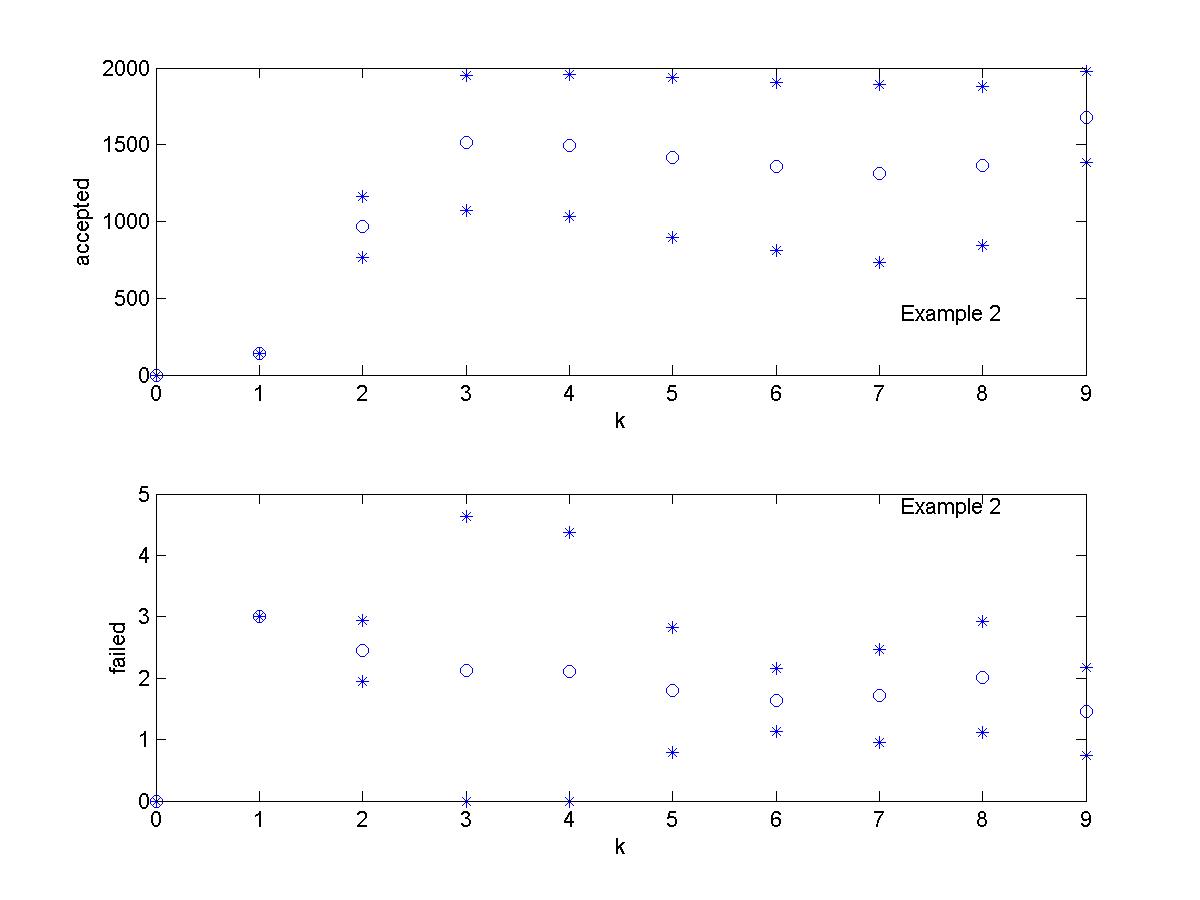}\\
\includegraphics[width=3.2in]{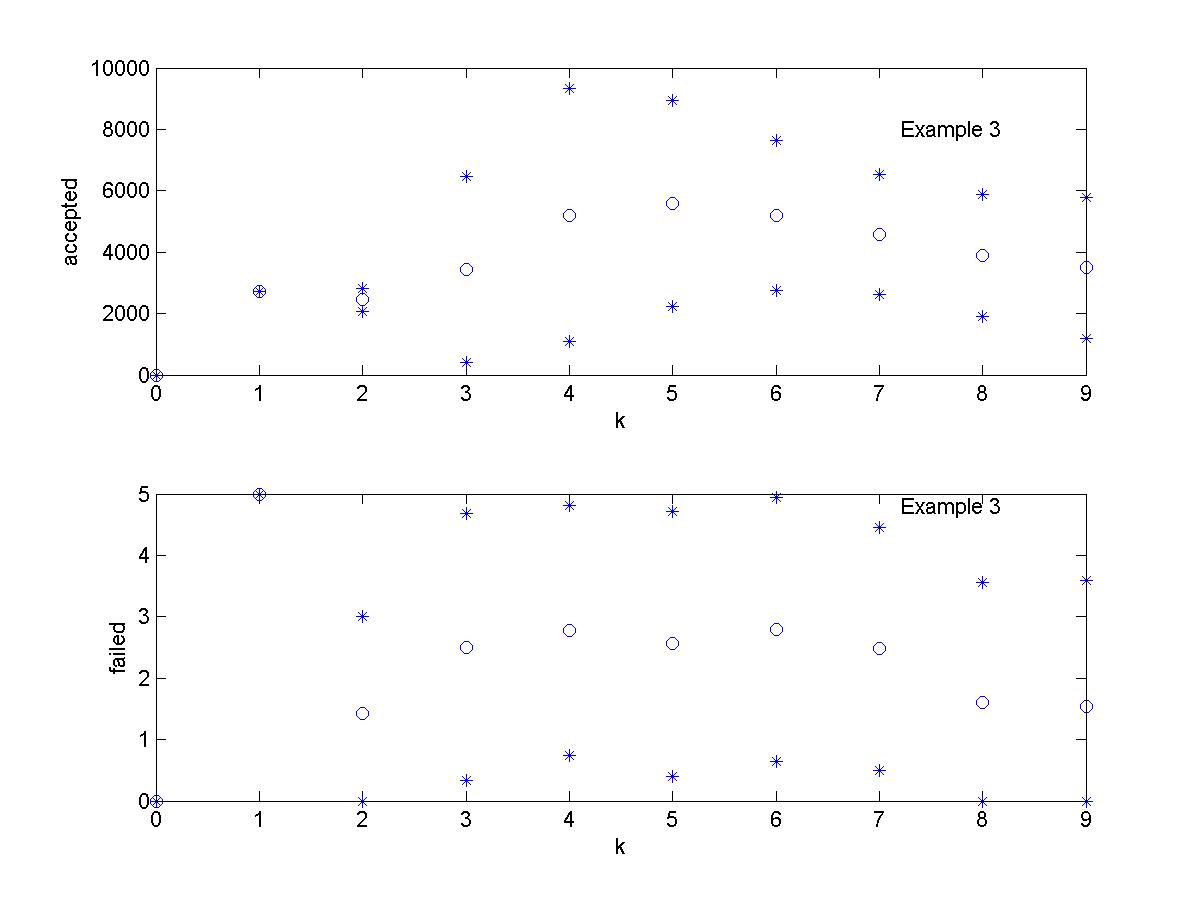}
\includegraphics[width=3.2in]{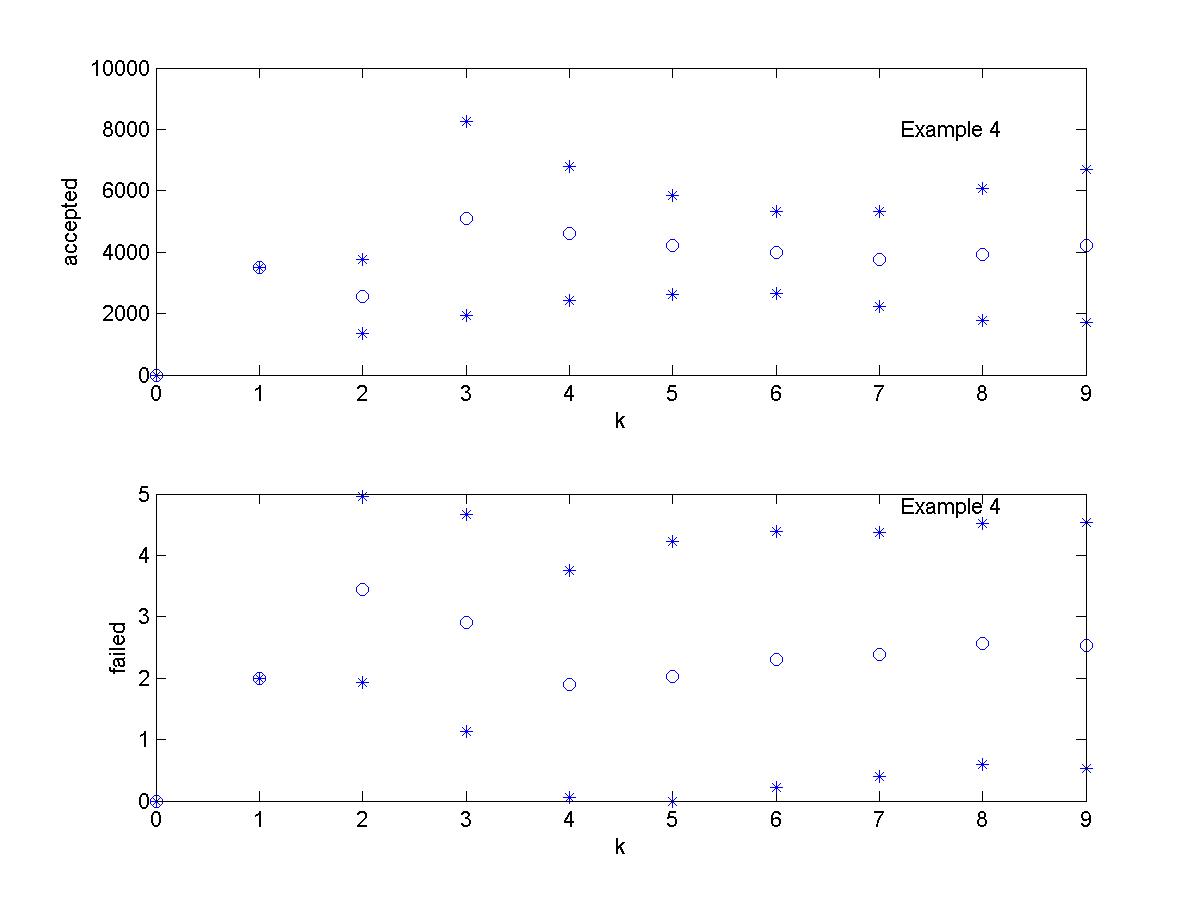}
\end{array}
$\newline\caption{Average (o) and 90\% confidence limits (*) of
accepted and failed steps of the adaptive LL filter
at each $t_{k}\in\{t\}_{M}$ in the four examples.}%
\label{Fig1}
\end{figure}

\subsection{Supplementary simulations}
As mentioned above, the approximate LMV filters play a central role
in the effective implementation of the innovation method for the
parameter estimation of diffusion processes given a set of partial
and noisy observations. Recently, in \cite{Jimenez AI}, the
performance of the innovation method based on different
approximations to the LMV filter has been evaluated by means of
simulations. In that paper, the parameters of the four state space
models considered in this section were estimated. The results show
that the estimators based on the order-$\beta$ LMV filters are
significantly more unbiased and efficient than the estimators based
on conventional approximations to the LMV filter, which clearly
illustrate the relevance of the approximate filters introduced here.
The reader interested in this type of identification problem is
encouraged to consider these simulations.

\section{Conclusions}

Approximate Linear Minimum Variance filters for continuous-discrete
state space models were introduced and their order of convergence is
stated. As particular instance, the order-$\beta$ Local
Linearization filters were studied in detail. For them, practical
algorithms were also provided and their performance in simulation
illustrated with various examples. Simulations show that: 1) with
thin time discretizations between observations, the order-$1$ LL
filter provides accurate approximations to the exact LMV filter; 2)
the convergence of the order-$1$ LL filter to the exact LMV filter
when the maximum stepsize of the time discretization between
observations decreases; 3) with respect to the conventional LL
filter, the order-$1$ LL filter significantly improves the
approximation to the exact LMV filter; 4) with an adequate
tolerance, the adaptive LL filter provides an automatic, accurate
and computationally efficient approximation to the LMV filtering
problem; and 5) the effectiveness of the order-$1$ LL filter for the
accurate identification of nonlinear stochastic systems from a
reduced number of partial and noisy observations distant in time.
Finally, it is worth noting that the approximate filters introduced
here have already been used in \cite{Jimenez AI} for the
implementation of computational efficient parameter estimators of
diffusion processes from partial and noisy observations, which would
have a positive impact in a variety of applications. Further, they
could be easily extended to deal with network-induced phenomena
(i.e., missing measurements and communication delays as considered
in \cite{Riera07,Hu12a,Hu12b}), which is currently a hot research
topic.

\bigskip

\section*{Acknowledgments}
This work was concluded on July 2012 within the framework of the
Associateship Scheme of the Abdus Salam International Centre for
Theoretical Physics (ICTP), Trieste, Italy. The author thanks to the
ICTP for the partial support to this work.

\end{document}